\documentclass[11pt,leqno]{article}

\usepackage[T1]{fontenc}
\usepackage{lmodern}
\usepackage[a4paper,textwidth=405bp,textheight=646bp,centering]{geometry}
\usepackage{amsmath,amssymb,amsthm,mathtools,mathrsfs,bm}
\usepackage{enumitem}
\usepackage{aliascnt}
\usepackage{indentfirst}
\usepackage{microtype}
\usepackage{xcolor}
\usepackage[numbers,sort&compress]{natbib}
\usepackage[colorlinks=true,linkcolor=blue,citecolor=blue,urlcolor=blue]{hyperref}
\usepackage[nameinlink,noabbrev]{cleveref}

\allowdisplaybreaks
\numberwithin{equation}{section}

\newtheorem{theorem}{Theorem}[section]

\newaliascnt{proposition}{theorem}
\newtheorem{proposition}[proposition]{Proposition}
\aliascntresetthe{proposition}

\newaliascnt{lemma}{theorem}
\newtheorem{lemma}[lemma]{Lemma}
\aliascntresetthe{lemma}

\newaliascnt{corollary}{theorem}

\aliascntresetthe{corollary}

\newaliascnt{definition}{theorem}
\newtheorem{definition}[definition]{Definition}
\aliascntresetthe{definition}

\newaliascnt{assumption}{theorem}
\newtheorem{assumption}[assumption]{Assumption}
\aliascntresetthe{assumption}

\newaliascnt{problem}{theorem}

\aliascntresetthe{problem}

\theoremstyle{remark}
\newaliascnt{remark}{theorem}
\newtheorem{remark}[remark]{Remark}
\aliascntresetthe{remark}

\crefname{theorem}{theorem}{theorems}
\crefname{proposition}{proposition}{propositions}
\crefname{lemma}{lemma}{lemmas}
\crefname{corollary}{corollary}{corollaries}
\crefname{definition}{definition}{definitions}
\crefname{assumption}{assumption}{assumptions}
\crefname{problem}{problem}{problems}
\crefname{remark}{remark}{remarks}

\newcommand{\R}{\mathbb R}
\newcommand{\N}{\mathbb N}
\newcommand{\Pp}{\mathbb P}
\newcommand{\Q}{\mathbb Q}
\newcommand{\E}{\mathbb E}
\newcommand{\F}{\mathcal F}
\newcommand{\FY}{\mathcal F^Y}
\newcommand{\cE}{\mathcal E}
\newcommand{\cM}{\mathcal M}

\newcommand{\dd}{\,\mathrm d}
\newcommand{\one}{\mathbf 1}
\newcommand{\ip}[2]{\left\langle #1,#2\right\rangle}
\newcommand{\norm}[1]{\left\lVert #1\right\rVert}
\newcommand{\abs}[1]{\left\lvert #1\right\rvert}
\newcommand{\br}[1]{\left\langle #1\right\rangle}
\newcommand{\esssup}{\operatorname*{ess\,sup}}
\newcommand{\supp}{\operatorname{supp}}

\newcommand{\sgn}{\operatorname{sgn}}

\title{Uniqueness and stability of nonlinear filtering equations with unbounded random coefficients
\thanks{ This work is supported by
		the National Key R\&D Program of China (2022YFA1006102),
		the National Natural Science Foundation of China (12471418, 12595294, 12231002),
		and the New Cornerstone Science Foundation (NCI202501).}}
\author{Jie Xiong
\thanks{Department of Mathematics and Shenzhen International Center for Mathematics,
		Southern University of Science and Technology, Shenzhen 518055, China.
		(\texttt{xiongj@sustech.edu.cn}).}
		 \and 
		 Wen Xu
\thanks{School of Mathematical Sciences, Peking University,
		Beijing 100871, China.
		(\texttt{xuwen@math.pku.edu.cn}).}
		 \and 
		 Ying Yang
\thanks{Department of Mathematics,
		Southern University of Science and Technology, Shenzhen 518055, China.
		(\texttt{12331007@mail.sustech.edu.cn}).}}
\date{}
\hypersetup{
  pdftitle={Uniqueness of nonlinear filtering equations with unbounded random coefficients},
  pdfauthor={Jie Xiong, Wen Xu, Ying Yang},
  pdfsubject={Nonlinear filtering, dual BSPDEs, and measure-valued uniqueness},
  pdfkeywords={nonlinear filtering, Zakai equation, BSPDE, unbounded observation drift}
}

\begin{document}
\maketitle

\begin{abstract}

We study a multidimensional nonlinear filtering model whose coefficients depend on a given observation-adapted predictable process and whose observation drift may grow linearly in both the state and the random input. Due to the unboundedness of the observation drift, a global
reference measure is not available. To overcome this hurdle, a localized entropy argument is adapted to prove the stopped likelihood to be a uniformly integrable martingale at each control-energy stopping level. The stopped Zakai equation, and hence, the 
stopped filtering equation is derived. The global  filtering equation is then established by de-localization. The uniqueness of the solution to
the stopped Zakai equation is obtained by a duality backward stochastic partial differential equation. This uniqueness then propagates to that of the global filtering equation through the stopped ones. Finally, a stability result is established in $W_1$-distance of measures.

\end{abstract}



%\medskip
%\noindent\textbf{Funding.}
%This work was supported by the National Key R\&D Program of China
%(2022YFA1006102) and the National Natural Science Foundation of China
%(12471418).

%\tableofcontents

\section{Introduction}
\label{sec:introduction}

Nonlinear filtering concerns the conditional law of an unobserved signal given a noisy observation.  For diffusion models, the normalized conditional distribution is governed by the Kushner--Stratonovich equation, while its unnormalized counterpart satisfies the linear Zakai equation.  The latter is often the more convenient object for analysis because of the linearity; see, among many standard references, \citet{Kallianpur1980,BainCrisan2009,Xiong2008}.  A basic question is whether these measure-valued equations determine the filter uniquely when their coefficients are random and unbounded.

We study this question for the multidimensional system
\begin{equation*}
 \begin{aligned}
  \dd X_t&=b(X_t,u_t)\dd t+\sigma(X_t,u_t)\dd W_t,\\
  \dd Y_t&=h(X_t,u_t)\dd t+\dd V_t,
 \end{aligned}
 \label{eq:intro-model}
\end{equation*}
where $X\in\R^d$, $W\in\R^r$, $Y,V\in\R^m$, the signal and observation noises are independent, and $u$ is an $\mathbb F^Y$-predictable process satisfying only the pathwise energy condition
\[
 \int_0^T |u_s|^2\dd s<\infty.
\]
The process $u$ need not be of the Markovian form $U(t,Y_t)$; it may depend predictably on the entire observation history.  Consequently, after conditioning on the observations, the coefficients in the filtering equations may form a non-Markovian, observation-adapted random environment.

Throughout the paper the admissible system $(X,Y,u)$ is taken as given.  
  At the same time, we allow
\[
 |b(x,v)|+\|\sigma(x,v)\|+|h(x,v)|\le C(1+|x|+|v|),
\]
so that the observation drift is unbounded and the second-order coefficient $a=\sigma\sigma^\top$ may grow quadratically in space.  For the full coefficient class, the likelihood stochastic exponential need not be covered by a bounded-coefficient Novikov argument, and the backward equation required by measure-valued duality may have spatially unbounded random coefficients, including the coefficient multiplying its martingale integrand.

There is a substantial literature on uniqueness of filtering equations.
Bensoussan \citep{Bensoussan1992} established measure-valued uniqueness for classical Zakai equations through duality with deterministic backward partial differential equations (PDEs), allowing degenerate signal diffusion and linear growth of the signal drift and observation function under bounded first- and second-order spatial derivatives.  Baras et al.  \cite{BarasBlankenshipHopkins1983} established path-by-path existence, uniqueness, and tail estimates for scalar density solutions of a class of Zakai equations with strongly unbounded deterministic Markov coefficients. Their argument transforms the forward Zakai equation into a classical robust parabolic PDE for each fixed observation path and applies coefficient-adapted exponential weights and maximum-principle techniques. Their principal results assume a nondegenerate signal diffusion and obtain uniqueness in a prescribed weighted decay class, a special bilinear degeneracy is treated after a logarithmic change of variables.
	
	The present problem is different in both the source of randomness and the solution concept. The coefficients \( b(\cdot, u_t), \sigma(\cdot, u_t) \), and \( h(\cdot, u_t) \) are generated by an arbitrary observation-predictable input and need not be functions of the current observation alone. Consequently, the backward dual object cannot be treated as a pathwise classical PDE: it is an adapted backward stochastic
partial differential equation (BSPDE) with the additional unknown \( g \) and the spatially unbounded coupling \( h(\cdot, u_t)^\top g_t \). Our uniqueness result is measure-valued, allows singular initial laws and degenerate signal diffusion, and is proved in an admissible class specified by a class-D mass condition and one finite polynomial moment.

Backward-equation duality for linearly growing coefficients, without an ellipticity requirement but with bounded first and second spatial derivatives, already appears in Bensoussan's treatment of partially observed systems \citep{Bensoussan1992}.  Unbounded-coefficient versions of the Zakai equation were studied by \citet{BarasBlankenshipHopkins1983} and \citet{Florchinger1993}.  The pathwise-uniqueness analysis of \citet{LucicHeunis2001} treats observation-conditioned signals by a different route, while the filtered-martingale-problem method of \citet{KurtzOcone1988} supplies an alternative uniqueness mechanism that does not proceed through the present whole-space BSPDE.

In the Markov setting, \citet{BKK} proved uniqueness and robustness for measure-valued Zakai and Fujisaki--Kallianpur--Kunita equations and allowed a continuous, possibly unbounded observation function.  Their signal is characterized by a deterministic Markov generator, so the backward dual object does not face the observation-adapted random environment considered here.  More recently, \citet{CrisanPardoux2026} established measure-valued uniqueness in a considerably more general signal--observation geometry, allowing observation-dependent coefficients, correlated Brownian noises, and degenerate observation diffusion.  Their uniqueness theory is based on a BSPDE duality argument and, under their Assumption U, imposes boundedness on the relevant zero-order coefficients together with high-order spatial regularity; their derivation of the filtering equations also isolates the martingale property of the likelihood as a separate assumption.

The present result is complementary rather than more general in noise geometry.  Its specific obstruction is the combination of a non-Markovian observation-predictable input, the resulting random dual coefficients, linear spatial growth of the zero-order coefficient multiplying the BSPDE martingale integrand, and the absence of density or ellipticity assumptions.  The analytical ingredient developed here is  a product formula valid for the stated admissible measure class.  The entropy estimate is used as a localized Bene\v{s}-type ingredient in that construction rather than as a claim of an unstopped likelihood theorem.

The first contribution is probabilistic and explicitly local in the control-energy level.  With $\tau_k$ being defined by \eqref{eq:tau-k},
we prove that, for each fixed $k$, the stopped likelihood is a uniformly integrable martingale by an entropy localization argument.  Under \Cref{ass:filtering}, the martingale step uses the linear-growth bounds, the pathwise control-energy bound, and the resulting stopped signal moment estimate; it does not require bounded $h$, a Novikov exponential moment, or BMO control.  This yields a levelwise reference probability $\Q^k$ and the stopped Zakai and Kushner--Stratonovich equations.  The family $(\Q^k)$ is consistent on overlapping stopped $\sigma$-fields, but no single global reference probability is asserted under the present assumptions.

The second contribution is the measure-valued duality step.  We prove a localized stochastic product formula for an admissible Zakai solution $\mu$ and the random field $(f,g)$.  Mollification and spatial cutoff produce only one nontrivial large-space commutator,
\[
 L_t(\chi_Rf)-\chi_RL_tf,
\]
whose highest growth order is $\langle x\rangle^{\lambda+1}$.  
Accordingly, the duality requires one local polynomial moment of order $q_0\ge\lambda+1$, together with a class-$D$ condition on the mass process.  A completely explicit admissible choice used below is $\lambda=d+3$ and $q_0=d+4$.  The resulting uniqueness statements are within the admissible classes of \Cref{def:admissible-zakai,def:admissible-KS}; the class-$D$, local-moment, and local observation-drift conditions are part of the conclusions' scope.
  The resulting identity
\[
 \langle\mu_{t\wedge\theta},f_{t\wedge\theta}\rangle
 =\langle\mu_0,f_0\rangle
 +\int_0^{t\wedge\theta}
   \langle\mu_s,g_s+f_sh(\cdot,u_s)\rangle^\top\dd Y_s
\]
permits random terminal tests.  Choosing
\[
 \gamma(\omega,x)=\operatorname{sgn}\langle\mu^1_\theta-\mu^2_\theta,\varphi\rangle\,\varphi(x)
\]
then yields pathwise uniqueness of the stopped Zakai equation within \Cref{def:admissible-zakai}'s admissible class.  Under the finite-$(d+4)$-moment hypothesis on the initial law, the actual stopped filters belong to this class and are consistent across levels.  An explicit unnormalization calculation transfers the stopped uniqueness to the Kushner--Stratonovich equation, and the control-energy stopping times increase to $T$, giving global uniqueness only for the physical Kushner--Stratonovich equation and for the consistent stopped Zakai family.

Finite-horizon robustness of nonlinear filters under model perturbations is classical; see \citet{BhattKallianpurKarandikar1999}, who prove continuity in the law of a possibly non-Markov signal for independent observation noise.  We record a common-reference criterion tailored to perturbations of the present observation-adapted input.  The comparison is made on a common reference space because the physical probability and the observation model vary with the input.  If the state processes converge in probability in $C([0,T];\R^d)$, the induced observation drifts $h(X_t^n,u_n(t))$ converge in the natural $L^2([0,T])$ sense in probability, and the terminal likelihoods are uniformly integrable, then, for every fixed time, the normalized filters converge in probability in bounded--Lipschitz distance.  Under a uniform second-moment bound for the states under their corresponding physical probabilities, the convergence holds in $W_1$ and therefore against every continuous test function of at most linear growth.  The theorem is stated in terms of the induced observation drifts because the present assumptions impose only Borel dependence on the input variable, so convergence of the inputs alone does not in general control the likelihoods.

The paper does not treat correlated signal--observation noises or nonidentity observation covariance.  In those models the Zakai noise operator contains spatial derivatives and the dual BSPDE acquires corresponding $Dg$ couplings, so the extension is analytical rather than notational.  The present independent-noise model isolates the difficulty relevant to observation-adapted random inputs and unbounded observation drift.  This form is also suited to later applications in partially observed stochastic control and related problems, where the auxiliary input is itself observation-adapted.

The rest of the paper is organized as follows.  \Cref{sec:filtering} constructs the localized reference probability and derives the stopped filtering equations under the minimal filtering assumptions.   \Cref{sec:measure-duality} establishes the measure-valued stochastic product formula, proves uniqueness of the stopped Zakai equations. The uniqueness of the stopped  Kushner--Stratonovich equations  is presented
in \Cref{KS-equation}. We also verify the consistency of the actual stopped filters, and finally patches the physical Kushner--Stratonovich uniqueness along the control-energy stopping times. \Cref{sec:stability} proves stability of the normalized filters in the weak and $W_1$ topologies under the common-reference likelihood conditions described above. The article is concluded in \Cref{sec:conclusion}. An appendix on BSPDE is presented at the end for the convenience of the reader.

%\Cref{sec:duality-problem} identifies the BSPDE forced by the Zakai duality.   
\section{Filtering model, localization, and the stopped equations}
\label{sec:filtering}

\subsection{Probabilistic setting and minimal filtering assumptions}

Let $T>0$ and let $(\Omega,\F,\mathbb F,\Pp)$ satisfy the usual conditions.
The filtration $\mathbb F=(\F_t)_{0\le t\le T}$ is the usual augmentation of the filtration generated by an $\R^r$-valued Brownian motion $W$, an $\R^m$-valued Brownian motion $V$, and an initial random variable $X_0$.  $W$ and $V$ are independent, and $X_0$ is independent of $(W,V)$.  
%Thus $\mathbb F$ has the predictable representation property with respect to $(W,V)$ after adjoining the initial sigma-field $\F_0=\sigma(X_0)\vee\mathcal N$. 
 We write $\nu_0$ for the law of the square-integrable random variable  $X_0$,  and use the convention $\inf\varnothing=+\infty$, and the notation  $
 \br{x}:=(1+|x|^2)^{1/2}.
 $ 

Let $U$ be a Euclidean space.  We work with a given \emph{admissible filtering system} $(X,Y,u)$ on this probability space: $Y_0=0$,
\begin{eqnarray*}
    \dd X_t
      &=&b(X_t,u_t)\dd t+\sigma(X_t,u_t)\dd W_t,
      \qquad X_t\in\R^d,
      \label{eq:state-physical}\\
    \dd Y_t
      &=&h(X_t,u_t)\dd t+\dd V_t,
      \qquad Y_t\in\R^m,
      \label{eq:obs-physical}
\end{eqnarray*}
$u$ is $\mathbb F^Y$-predictable for the usual augmentation of the filtration generated by $Y$, and
\begin{equation}
    \int_0^T \abs{u_s}^2\dd s<\infty,
    \qquad \Pp\text{-a.s.}
    \label{eq:u-energy}
\end{equation}
 The formulation does not require $u_t$ to be a Markovian function of $Y_t$, it may depend predictably on the entire observation history. 
\begin{assumption}
\label{ass:filtering}
The functions
\[
 b:\R^d\times U\to\R^d,
 \qquad
 \sigma:\R^d\times U\to\R^{d\times r},
 \qquad
 h:\R^d\times U\to\R^m
\]
are Borel measurable.  There is a constant $K\ge1$ such that, for all $x,y\in\R^d$ and $v\in U$,
\begin{align}
 &\abs{b(x,v)-b(y,v)}+\norm{\sigma(x,v)-\sigma(y,v)}
   \le K\abs{x-y},
   \label{eq:filtering-lipschitz}\\
 &\abs{b(x,v)}+\norm{\sigma(x,v)}+\abs{h(x,v)}
   \le K(1+\abs{x}+\abs{v}).
   \label{eq:filtering-growth}
\end{align}
\end{assumption}

For $k\in\N$, define
\begin{equation}
    \tau_k
      :=\inf\left\{t\in[0,T]:\int_0^t\abs{u_s}^2\dd s\ge k\right\}\wedge T.
    \label{eq:tau-k}
\end{equation}
Then 
$\tau_k$ is an $\mathbb F^Y$-stopping time, and  $\tau_k\uparrow T$ almost surely by \eqref{eq:u-energy}.

\begin{lemma}[Stopped signal moments]
\label{lem:signal-moments}
Suppose \Cref{ass:filtering} holds.  Let $q\ge2$, and let $\widetilde{\Pp}$ be a probability measure equivalent to $\Pp$ such that $W$ remains an $\R^r$-valued Brownian motion under $\widetilde{\Pp}$.  If
\[
    \E^{\widetilde{\Pp}}|X_0|^q<\infty,
\]
then
\begin{equation}
  \E^{\widetilde{\Pp}}\sup_{0\le s\le\tau_k}\br{X_s}^{q}
  \le C_{k,q,T,K}\left(1+\E^{\widetilde{\Pp}}\abs{X_0}^{q}\right).
  \label{eq:signal-moment}
\end{equation}
\end{lemma}

\begin{proof}

Set $\eta_n:=\inf\{t\in[0,T]:|X_t|\ge n\}\wedge T$ and first stop the
state equation at $\tau_k\wedge\eta_n$.  In the estimates displayed below,
each occurrence of $\tau_k$ is temporarily understood as
$\tau_k\wedge\eta_n$.  The linear-growth bound, H\"older's inequality in
time, and the Burkholder--Davis--Gundy inequality then apply without assuming
in advance the finiteness of the quantity being estimated, and their
constants are independent of $n$.

\begin{align*}
 \E^{\widetilde{\Pp}}\sup_{r\le t\wedge\tau_k}|X_r|^q
 \le C_q\Bigg(&\E^{\widetilde{\Pp}}|X_0|^q
 +\E^{\widetilde{\Pp}}\Big(\int_0^{t\wedge\tau_k}
        (1+|X_s|+|u_s|)\dd s\Big)^q\\
 &+\E^{\widetilde{\Pp}}\Big(\int_0^{t\wedge\tau_k}
        (1+|X_s|+|u_s|)^2\dd s\Big)^{q/2}\Bigg).
\end{align*}
The control terms are bounded pathwise by
\[
 \int_0^{\tau_k}|u_s|\dd s\le T^{1/2}k^{1/2},
 \qquad
 \left(\int_0^{\tau_k}|u_s|^2\dd s\right)^{q/2}\le k^{q/2}.
\]
For the state terms, H\"older's inequality yields
\[
 \left(\int_0^{t\wedge\tau_k}|X_s|\dd s\right)^q
 +\left(\int_0^{t\wedge\tau_k}|X_s|^2\dd s\right)^{q/2}
 \le C_{q,T}\int_0^t\sup_{a\le s\wedge\tau_k}|X_a|^q\dd s.
\]
Consequently,
\[
 \E^{\widetilde{\Pp}}\sup_{r\le t\wedge\tau_k}|X_r|^q
 \le C_{k,q,T,K}\left(1+\E^{\widetilde{\Pp}}|X_0|^q
 +\int_0^t\E^{\widetilde{\Pp}}\sup_{a\le s\wedge\tau_k}|X_a|^q\dd s\right).
\]
Gronwall's lemma gives the bound uniformly in $n$.  Letting
$n\uparrow\infty$ and applying Fatou's lemma proves
\eqref{eq:signal-moment}.  Notice that only the $L^2$-energy of $u$ is used; no $L^q$-in-time assumption on $u$ is required.  See \citet[Chapter~2]{KaratzasShreve1991} for the standard SDE moment argument.
\end{proof}

\subsection{The localized likelihood is a true martingale}

Define 
$
 H_t:=h(X_t,u_t)
$ 
and the stopped likelihood
\begin{equation*}
 \Lambda_t^k
 :=\cE\left(-\int_0^{\,\cdot\wedge\tau_k}H_s^\top\dd V_s\right)_t
 =\exp\left(
   -\int_0^{t\wedge\tau_k}H_s^\top\dd V_s
   -\frac12\int_0^{t\wedge\tau_k}|H_s|^2\dd s
 \right).
 \label{eq:Lambda-k}
\end{equation*}
The next proposition replaces a Novikov assumption by an entropy estimate.  This is a localized Bene\v{s}-type mechanism; see \citet{KlebanerLiptser2014} for the linear-growth exponential-martingale principle and \citet{CassClarkCrisan2014} for its role in reference-probability filtering.

\begin{proposition}[Entropy criterion for the filtering likelihood]
\label{prop:filtering-likelihood}
Under \Cref{ass:filtering}, $\Lambda^k$ is a uniformly integrable $\Pp$-martingale.  In particular,
\begin{equation*}
  \E^{\Pp}\Lambda_t^k=1,
  \qquad 0\le t\le T.
  \label{eq:Lambda-expectation}
\end{equation*}
\end{proposition}

\begin{proof}
For $j\in\N$, set
\begin{equation*}
 \rho_j
 :=\inf\left\{t\in[0,T]:
       \int_0^{t\wedge\tau_k}|H_s|^2\dd s\ge j\right\}\wedge\tau_k
 \label{eq:rho-j}
\end{equation*}
and
\[
 \Lambda_t^{k,j}
 :=\cE\left(-\int_0^{\,\cdot\wedge\rho_j}H_s^\top\dd V_s\right)_t.
\]
Since the quadratic variation of its stochastic logarithm is bounded by $j$, Novikov's criterion implies that $\Lambda^{k,j}$ is a uniformly integrable martingale.  Define $\Pp^j$ by
\begin{equation*}
 \frac{\dd\Pp^j}{\dd\Pp}\bigg|_{\F_T}=\Lambda_T^{k,j}=\Lambda_{\rho_j}^{k,j}.
 \label{eq:Pj}
\end{equation*}
The multidimensional Girsanov theorem gives that
\[
 \widehat V_t^j:=V_t+\int_0^{t\wedge\rho_j}H_s\dd s
\]
is an $\R^m$-valued Brownian motion under $\Pp^j$.  The density is driven only by $V$ and $[W,V]=0$, hence $W$ remains an $\R^r$-valued Brownian motion under $\Pp^j$. It is also easy to show that the initial distribution is preserved. The state equation is therefore unchanged under $\Pp^j$.

Applying \Cref{lem:signal-moments} under $\Pp^j$ yields
\begin{equation}
 \sup_{j\ge1}\E^{\Pp^j}
       \sup_{0\le s\le\tau_k}|X_s|^2
 \le C_k\bigl(1+\E^{\Pp}|X_0|^2\bigr).
 \label{eq:Pj-state-bound}
\end{equation}
By \eqref{eq:filtering-growth}, \eqref{eq:tau-k}, and \eqref{eq:Pj-state-bound},
\begin{equation*}
 \sup_{j\ge1}\E^{\Pp^j}
       \int_0^{\rho_j}|H_s|^2\dd s
 \le C_k\bigl(1+\E^{\Pp}|X_0|^2\bigr).
 \label{eq:Pj-H-bound}
\end{equation*}
Under $\Pp^j$,
\[
 \log\Lambda_{\rho_j}^{k,j}
 =-\int_0^{\rho_j}H_s^\top\dd\widehat V_s^j
   +\frac12\int_0^{\rho_j}|H_s|^2\dd s.
\]
The stopped stochastic integral is square-integrable.  Hence
\begin{equation*}
 \E^{\Pp}\!
 \left[\Lambda_{\rho_j}^{k,j}\log\Lambda_{\rho_j}^{k,j}\right]
 =\frac12\E^{\Pp^j}\int_0^{\rho_j}|H_s|^2\dd s
 \le C_k\bigl(1+\E^{\Pp}|X_0|^2\bigr).
 \label{eq:filtering-entropy}
\end{equation*}
Together with $\E^{\Pp}\Lambda_{\rho_j}^{k,j}=1$, this bounds
$\E\Phi(\Lambda_{\rho_j}^{k,j})$ uniformly in $j$ for
$\Phi(z)=z\log z-z+1$.  Since $\Phi(z)/z\to\infty$, the de la Vall\'ee--Poussin criterion implies uniform integrability of
$\{\Lambda_{\rho_j}^{k,j}:j\ge1\}$.

The continuity of $X$ and \eqref{eq:tau-k} imply
$\int_0^{\tau_k}|H_s|^2\dd s<\infty$ a.s..  Thus $\rho_j\uparrow\tau_k$ and
$\Lambda_{\rho_j}^{k,j}\to\Lambda_T^k$ a.s..  Uniform integrability gives convergence in $L^1(\Pp)$ and
$\E^{\Pp}\Lambda_T^k=1$.  A nonnegative local martingale whose terminal expectation equals its initial value is closed by its terminal value.  Consequently
$\Lambda_t^k=\E^{\Pp}[\Lambda_T^k\mid\F_t]$ and $\Lambda^k$ is uniformly integrable.
\end{proof}

\subsection{Reference probability and the filtering equations}

The construction in this subsection is levelwise in $k$.  Proposition~\ref{prop:filtering-likelihood} proves the martingale property only for the exponential stopped at $\tau_k$.  Thus $\Q^k$ below is a local reference probability for identities stopped no later than $\tau_k$.

Define a probability measure $\Q^k$ on $\F_T$ by
\begin{equation*}
 \frac{\dd\Q^k}{\dd\Pp}\bigg|_{\F_T}=\Lambda_T^k.
 \label{eq:Qk}
\end{equation*}
Set
\begin{equation*}
 Y_t^k:=V_t+\int_0^{t\wedge\tau_k}H_s\dd s.
 \label{eq:Ybar-k}
\end{equation*}

Let $\mathcal N$ be the common collection of null sets under
$\Pp$ and $\Q^k$; the two collections agree because $\Q^k\sim\Pp$.
By Girsanov's theorem, $(W,Y^k)$ is an $(r+m)$-dimensional Brownian motion
under $\Q^k$.  Define
\begin{equation*}
 \begin{aligned}
\widetilde{\mathcal G}_t^k
 &:=\sigma(Y_s^k:0\le s\le t)\vee\mathcal N,\\
 \mathcal G_t^k
 &:=\sigma(Y^k_{s\wedge\tau_k}:0\le s\le t)\vee\mathcal N,
 \qquad 0\le t\le T.
 \end{aligned}
 \label{eq:completed-stopped-observation}
\end{equation*}
Thus $\widetilde{\mathbb G}^k=(\widetilde{\mathcal G}_t^k)_{t\le T}$ is the usual natural
filtration of the $\Q^k$-Brownian motion $Y^k$.  

\begin{lemma}[Observation-functional stopping and stopped PRP]
\label{lem:reference-stopped-prp}
\label{lem:reference-stopping-functional}
For each $k$, the stopping time $\tau_k$ has an $\widetilde{\mathbb G}^k$-stopping-time
version.  With that version,
\begin{equation}
 \mathcal G_t^k
 =\widetilde{\mathcal G}^k_{t\wedge\tau_k}
 =\FY_{t\wedge\tau_k},
 \qquad 0\le t\le T,
 \label{eq:stopped-observation-identity}
\end{equation}
where the $\sigma$-fields at stopping times are completed by $\mathcal N$.
Moreover, $\mathbb G^k$ has the predictable representation property (PRP) with
respect to $Y^k_{\cdot\wedge\tau_k}$: every $\mathbb G^k$-local martingale
$N$ can be written
\begin{equation}
 N_t=N_0+\int_0^{t\wedge\tau_k}\zeta_s^\top\dd Y_s^k
 \label{eq:stopped-reference-prp}
\end{equation}
for a $\mathbb G^k$-predictable integrand $\zeta$ which is locally square
integrable on $[0,\tau_k]$.
\end{lemma}

\begin{proof}
Let $\mathsf C_m=C([0,T];\R^m)$, write $y^s=y_{\cdot\wedge s}$, and equip
$\mathsf C_m$ with its raw coordinate filtration.  Since $u$ is predictable
for the usual augmentation of the natural filtration of $Y$, predictable
factorization supplies a canonical-predictable Borel map
\[
 \mathfrak u:[0,T]\times\mathsf C_m\longrightarrow U,
 \qquad
 \mathfrak u(s,y)=\mathfrak u(s,y^s),
\]
such that
\begin{equation}
 u_s(\omega)=\mathfrak u(s,Y(\omega))
 \quad \dd s\otimes\dd\Pp\text{-a.e.}
 \label{eq:canonical-control-factorization}
\end{equation}
The same equality holds $\dd s\otimes\dd\Q^k$-a.e. because
$\Q^k\sim\Pp$.  By Fubini, outside one common null set the two accumulated
energies agree for every time.

For $y\in\mathsf C_m$, define the  nondecreasing continuous
functional
\begin{equation*}
 A_t(y):=\int_0^t|\mathfrak u(s,y^s)|^2\dd s,
 \qquad
 \vartheta_k(y):=\inf\{t\in[0,T]:A_t(y)\ge k\}\wedge T.
 \label{eq:canonical-energy-hitting}
\end{equation*}
For $t<T$, the event $\{\vartheta_k\le t\}$ coincides with
$\{A_t\ge k\}$ and depends only on $y^t$; at $t=T$ it is the whole path
space.  Hence $\vartheta_k$ is a stopping-time  for the raw
coordinate filtration.  Equation
\eqref{eq:canonical-control-factorization} and Fubini give
\begin{equation}
 \tau_k=\vartheta_k(Y)
 \quad \Pp\text{- and }\Q^k\text{-a.s.}
 \label{eq:tau-as-Y-functional}
\end{equation}

Put $\tau=\vartheta_k(Y)$ and $\tau'=\vartheta_k(Y^k)$ on the full-probability
set on which \eqref{eq:tau-as-Y-functional} holds.  By the definition of
$Y^k$, the paths $Y$ and $Y^k$ agree on $[0,\tau]$.  Nonanticipativity of
$\mathfrak u$ therefore gives
\[
 A_t(Y^k)=A_t(Y),\qquad 0\le t\le\tau.
\]
It is then easy to see that
\begin{equation*}
 \tau_k=\vartheta_k(Y^k)\qquad \Q^k\text{-a.s.}
 \label{eq:tau-as-Yk-functional}
\end{equation*}
Replacing $\tau_k$ on a common null set by the right-hand side makes it an
$\widetilde{\mathbb G}^k$-stopping time without changing any stopped process.

Galmarino's test may now be applied in the reference Brownian filtration and
gives
\[
 \sigma(Y^k_{s\wedge\tau_k}:0\le s\le t)\vee\mathcal N
 =\widetilde{\mathcal G}^k_{t\wedge\tau_k}.
\]
On the other hand, $Y^k_{\cdot\wedge\tau_k}
=Y_{\cdot\wedge\tau_k}$.  Applying the same stopped-path identity in the
usual natural filtration of the physical observation gives
\eqref{eq:stopped-observation-identity}, including the common completion.

It remains to verify the representation assertion.  If
$\xi\in L^2(\mathcal G_T^k,\Q^k)=L^2(\widetilde{\mathcal G}^k_{\tau_k},\Q^k)$ and
\[
 \widetilde N_t:=\E^{\Q^k}[\xi\mid\widetilde{\mathcal G}_t^k],
\]
the Brownian predictable representation property in $\widetilde{\mathbb G}^k$ gives
\[
 \widetilde N_t=\widetilde N_0+\int_0^t\zeta_s^\top\dd Y_s^k.
\]
Optional sampling and \eqref{eq:stopped-observation-identity} yield
\[
 \E^{\Q^k}[\xi\mid\mathcal G_t^k]
 =\widetilde N_{t\wedge\tau_k}
 =\widetilde N_0+
   \int_0^{t\wedge\tau_k}\zeta_s^\top\dd Y_s^k.
\]
The stopped integrand has a $\mathbb G^k$-predictable version.  This proves
the representation for square-integrable martingales; localization and
pasting of the stopped integrands prove \eqref{eq:stopped-reference-prp} for
local martingales.
\end{proof}

The same factorization also fixes the predictable version of the stopped
control.  Namely,
\begin{equation}
 \one_{\{s\le\tau_k\}}u_s
 =\one_{\{s\le\tau_k\}}\mathfrak u(s,Y^k)
 \quad \dd s\otimes\dd\Q^k\text{-a.e.},
 \label{eq:stopped-control-factorization}
\end{equation}
and the right-hand side has a $\widetilde{\mathbb G}^k$-predictable version because
$\mathfrak u(\cdot,Y^k)$ is $\widetilde{\mathbb G}^k$-predictable and is stopped at the
$\widetilde{\mathbb G}^k$-stopping time $\tau_k$.  By equivalence the same statement
holds under $\Pp$.

Every fixed-$k$ conditional expectation, optional or predictable projection,
and stochastic integral below is taken in $\mathbb G^k$ and is stopped no
later than $\tau_k$.  We therefore write $\Q$ for $\Q^k$ and, only inside
such stopped identities, write $Y$ for $Y^k$.

If $\ell\ge k$ and $A\in\F_{\tau_k}$, where $\F_{\tau_k}$ is the completed stopped $\sigma$-field, optional sampling for the density martingale gives
\begin{align}
 \Q^\ell(A)
 &=\E^{\Pp}\!\left[\one_A\Lambda_T^\ell\right]
  =\E^{\Pp}\!\left[\one_A\Lambda_{\tau_k}^\ell\right]
  =\E^{\Pp}\!\left[\one_A\Lambda_{\tau_k}^k\right]
  =\Q^k(A).
 \label{eq:Q-consistency}
\end{align}
Indeed, the two stochastic exponentials have the same integrand up to $\tau_k$.  Hence
$\Q^\ell|_{\F_{\tau_k}}=\Q^k|_{\F_{\tau_k}}$ on the completed stopped $\sigma$-field; equivalence with $\Pp$ ensures that the completion is common to all the measures involved.

The change of measure also preserves the initial marginal: since $\Lambda^k$ is a uniformly integrable martingale,
$\E^{\Pp}[\Lambda_T^k\mid\F_0]=1$.  Thus $X_0$ has law $\nu_0$ under both $\Pp$ and $\Q^k$.

Let
\begin{equation*}
 M_t^k:=(\Lambda_t^k)^{-1}.
 \label{eq:M-k}
\end{equation*}
Then, $M^k$ is a uniformly integrable $\Q^k$-martingale and satisfies
\begin{equation}
 \dd M_t^k=M_t^kH_t^\top\one_{\{t\le\tau_k\}}\dd Y_t,
 \qquad M_0^k=1.
 \label{eq:M-k-SDE}
\end{equation}

Since $M^k_{\cdot\wedge\tau_k}$ is a class-$D$ process, the optional
regular-conditional-kernel theorem applied to the random finite-measure
process
\[
 M^k_{t\wedge\tau_k}\delta_{X_{t\wedge\tau_k}}(\dd x)
\]
gives a $\mathbb G^k$-optional kernel
$(t,\omega)\mapsto\mu_t^k(\omega,\dd x)$ with values in
$\mathcal M_+(\R^d)$.  It may be fixed on one countable
convergence-determining class containing $1$ and then extended by the
monotone-class theorem so that, for every bounded Borel $\varphi$,
\begin{equation}
 \ip{\mu_t^k}{\varphi}
 =\E^{\Q}\!\left[
     M^k_{t\wedge\tau_k}\varphi(X_{t\wedge\tau_k})
     \mid\mathcal G_t^k
   \right]
 \label{eq:mu-definition}
\end{equation}
as an optional-projection identity.  In particular, the corresponding
identity holds at every $\mathbb G^k$-stopping time.  All pairings below use
this one kernel, not separately selected scalar conditional expectations.

We next specify, from this same kernel, the predictable versions used under
time integrals.  First let
\begin{equation}
 \Psi_s(\omega,x)
 =\sum_{j=1}^J\xi_s^j(\omega)\varphi_j(x),
 \label{eq:predictable-simple-kernel}
\end{equation}
where each $\xi^j$ is bounded and $\mathbb G^k$-predictable and each
$\varphi_j$ is bounded Borel.  Define
\begin{equation}
 \one_{\{s\le\tau_k\}}\ip{\mu_s^k}{\Psi_s}
 :={}^{p,\mathbb G^k,\Q}\!\left(
       M_s^k\Psi_s(X_s)\one_{\{s\le\tau_k\}}
    \right),
 \label{eq:predictable-pairing}
\end{equation}
where the notation ${}^{p,\mathbb G^k,\Q}$ stands for $\Q$-predictable projection
with respect to the filtration $\mathbb G^k$.
For a simple kernel \eqref{eq:predictable-simple-kernel}, the pull-out
property of predictable projection gives, $\dd s\otimes\dd\Q$-a.e.,
\[
 {}^{p,\mathbb G^k,\Q}\!\left(
       M_s^k\Psi_s(X_s)\one_{\{s\le\tau_k\}}
    \right)
 =\sum_{j=1}^J\xi_s^j
   \one_{\{s\le\tau_k\}}\ip{\mu_s^k}{\varphi_j}.
\]
Indeed, $\mathbb G^k$ is a continuous stopped Brownian filtration by
\Cref{lem:reference-stopped-prp}; hence its optional and predictable
projections agree $\dd s\otimes\dd\Q$-a.e. for the absolutely
continuous clock $\dd s$.  Thus the right-hand side is precisely integration
of the simple random kernel against the fixed optional kernel.

The bounded simple kernels above generate
$\mathcal P(\mathbb G^k)\otimes\mathcal B(\R^d)$.  The functional monotone
class theorem therefore extends \eqref{eq:predictable-pairing} first to every
bounded nonnegative
$\mathcal P(\mathbb G^k)\otimes\mathcal B(\R^d)$-measurable $\Psi$, and then
by positive/negative decomposition to bounded signed or vector-valued
$\Psi$.  Truncation and localization give the same conclusion whenever
\[
 M_s^k|\Psi_s(X_s)|\one_{\{s\le\tau_k\}}
\]
is locally integrable.  We always choose the predictable projection in
\eqref{eq:predictable-pairing} to be zero on $(\tau_k,T]$.
By \eqref{eq:stopped-control-factorization}, the stopped kernels
$L_s\varphi$, $\varphi h(\cdot,u_s)$, and $h(\cdot,u_s)$ fall within this
construction.  The estimates in the proof of \Cref{thm:stopped-zakai} below
verify the required integrability for the first two; \Cref{lem:mass-process}
does so for the last.

For the moment write
\begin{equation*}
 \overline Z_t^k:=\ip{\mu_t^k}{1}.
 \label{eq:raw-mass}
\end{equation*}
Normalization is postponed until \Cref{lem:mass-process}, where the kernel
and its mass are modified on the same evanescent set so that the mass is one
strictly positive continuous process simultaneously for all times.

Put $a=\sigma\sigma^\top$ and define
\begin{equation*}
 L_t\varphi(x)
 :=\frac12a^{ij}(x,u_t)D_{ij}\varphi(x)
   +b^i(x,u_t)D_i\varphi(x).
 \label{eq:L}
\end{equation*}
Repeated state indices are summed from $1$ to $d$ and observation indices from $1$ to $m$.

\begin{theorem}[Stopped Zakai equation]
\label{thm:stopped-zakai}
For every $\varphi\in C_c^2(\R^d)$,
\begin{align}
 \ip{\mu_t^k}{\varphi}
  ={}&\ip{\nu_0}{\varphi}
  +\int_0^{t\wedge\tau_k}\ip{\mu_s^k}{L_s\varphi}\dd s
  +\sum_{a=1}^m\int_0^{t\wedge\tau_k}
       \ip{\mu_s^k}{\varphi h^a(\cdot,u_s)}\dd Y_s^a.
 \label{eq:zakai-stopped}
\end{align}
\end{theorem}

\begin{proof}

We first record the integrability which permits projection without an
$L^2(\Q)$ assumption on the density.  The growth bound, the pathwise control
energy bound, and \Cref{lem:signal-moments} under $\Pp$ give
\begin{equation}
 \E^{\Pp}\int_0^{\tau_k}|H_s|^2\dd s<\infty.
 \label{eq:H2-under-P}
\end{equation}
Under $\Pp$, using $\dd Y_s=H_s\dd s+\dd V_s$ on
$[0,\tau_k]$, we have
\[
 \log M^k_{\tau_k}
 =\int_0^{\tau_k}H_s^\top\dd V_s
   +\frac12\int_0^{\tau_k}|H_s|^2\dd s.
\]
The stochastic integral is square-integrable by \eqref{eq:H2-under-P} and has
mean zero.  Change of measure therefore yields
\begin{equation*}
 \E^{\Q}\!\left[M^k_{\tau_k}\log M^k_{\tau_k}\right]
 =\E^{\Pp}\!\left[\log M^k_{\tau_k}\right]
 =\frac12\E^{\Pp}\int_0^{\tau_k}|H_s|^2\dd s<\infty.
 \label{eq:M-LlogL}
\end{equation*}
The change-of-measure equality is legitimate also for the absolute value:
\[
\E^{\Q}[M^k_{\tau_k}|\log M^k_{\tau_k}|]
=\E^{\Pp}|\log M^k_{\tau_k}|<\infty.
\]
Doob's $L\log L$ maximal inequality now gives
\begin{equation}
 \E^{\Q}\sup_{0\le s\le\tau_k}M_s^k<\infty.
 \label{eq:M-H1}
\end{equation}

Fix $\varphi\in C_c^2(\R^d)$.  Compactness of the supports of
$\varphi$, $D\varphi$, and $D^2\varphi$, together with
\eqref{eq:filtering-growth} and \eqref{eq:tau-k}, gives the pathwise bound
\begin{equation}
 \int_0^{\tau_k}\!\left(
   |\varphi(X_s)H_s|^2
   +\|D\varphi(X_s)\sigma(X_s,u_s)\|^2
   +|L_s\varphi(X_s)|
 \right)\dd s
 \le C_{\varphi,k,T,K}.
 \label{eq:compact-test-energy}
\end{equation}
Indeed, each integrand on the left is bounded by a constant times
$1+|u_s|^2$ whenever it is nonzero.  Hence \eqref{eq:M-H1}, the
Burkholder--Davis--Gundy inequality, and \eqref{eq:compact-test-energy} imply
\begin{equation}
 \E^{\Q}\!\left(\int_0^{\tau_k}
       |M_s^k\varphi(X_s)H_s|^2\dd s\right)^{1/2}
 +\E^{\Q}\!\left(\int_0^{\tau_k}
       \|M_s^kD\varphi(X_s)\sigma(X_s,u_s)\|^2\dd s\right)^{1/2}
 <\infty,     \label{eq:two-H1}
 \end{equation}
 and
 \begin{equation*}   \E^{\Q}\int_0^{\tau_k}M_s^k|L_s\varphi(X_s)|\dd s<\infty.
 \label{eq:drift-L1}
\end{equation*}
Thus both stochastic integrals in the It\^o decomposition below are
$H^1(\Q)$ martingales, and its finite-variation part is integrable.

It\^o's formula, \eqref{eq:M-k-SDE}, and $[W,Y]=0$ give
\begin{align*}
 M^k_{t\wedge\tau_k}\varphi(X_{t\wedge\tau_k})
 ={}&\varphi(X_0)
 +\int_0^{t\wedge\tau_k}M_s^kL_s\varphi(X_s)\dd s\\
 &+\int_0^{t\wedge\tau_k}M_s^k\varphi(X_s)H_s^\top\dd Y_s
 +\int_0^{t\wedge\tau_k}
       M_s^kD\varphi(X_s)\sigma(X_s,u_s)\dd W_s.
\end{align*}
Take the optional projection onto $\mathbb G^k$ and the dual predictable
projection of the integrable drift.  By
\eqref{eq:predictable-pairing}, these projections are respectively
$\ip{\mu_s^k}{L_s\varphi}$ and
$\ip{\mu_s^k}{\varphi h(\cdot,u_s)}$.  The last $W$-integral has zero optional
projection.  Indeed, it is an $H^1$ martingale by \eqref{eq:two-H1}; every
$\mathbb G^k$-local martingale is an integral with respect to
$Y_{\cdot\wedge\tau_k}$ by the stopped PRP; and $[W,Y]=0$, so the $W$-integral
is strongly orthogonal to every $\mathbb G^k$-local martingale.  Testing
against bounded stopped $\mathbb G^k$-martingales, followed by the $H^1$
localization, shows that its optional projection is identically zero.

The standard optional-projection identity for an $H^1$ stochastic integral
with respect to the $\mathbb G^k$-Brownian motion gives
\[
 {}^{o,\mathbb G^k}\!\left(
   \int_0^{\,\cdot\wedge\tau_k}
       M_s^k\varphi(X_s)H_s^\top\dd Y_s
 \right)
 =\int_0^{\,\cdot\wedge\tau_k}
       \ip{\mu_s^k}{\varphi h(\cdot,u_s)}\dd Y_s.
\]
Equivalently, this identity follows by pairing both sides with bounded
$\mathbb G^k$-predictable elementary integrands and using the stopped PRP;
\eqref{eq:two-H1} permits removal of the localization.  Finally,
$\mathcal G_0^k$ is trivial up to $\mathcal N$, and $X_0$ has law $\nu_0$
under $\Q$.  We obtain \eqref{eq:zakai-stopped} with the predictable versions
specified in \eqref{eq:predictable-pairing}.

\end{proof}

Finally, we proceed to establishing the stopped filtering equation.

\begin{lemma}[Mass process]
\label{lem:mass-process}
The optional kernel $\mu^k$ can be modified on an evanescent set, and
$\overline Z^k$ can be replaced by an indistinguishable version $Z^k$, so
that
\[
 Z_t^k=\ip{\mu_t^k}{1}
\]
for all $t$ outside one fixed null set and $Z^k$ is a strictly positive
continuous $\mathbb G^k$-martingale.  With these synchronized versions,
\begin{equation}
 \pi_t^k(\dd x):=\frac{\mu_t^k(\dd x)}{Z_t^k}
 \label{eq:KS-formula}
\end{equation}
is a $\mathbb G^k$-optional probability kernel and, for every bounded Borel
$\varphi$,
\[
 \ip{\pi_t^k}{\varphi}
 =\E^{\Pp}\!\left[
     \varphi(X_{t\wedge\tau_k})\mid\mathcal G_t^k
   \right]
 \quad\text{a.s. for every }t.
\]
\end{lemma}

\begin{proof}
Because $M^k$ is stopped at $\tau_k$ and is closed by $M^k_{\tau_k}$, the
tower property and \eqref{eq:mu-definition} give
\begin{equation}
 \overline Z_t^k
 =\E^{\Q}[M^k_{t\wedge\tau_k}\mid\mathcal G_t^k]
 =\E^{\Q}[M^k_{\tau_k}\mid\mathcal G_t^k].
 \label{eq:mass-closed-projection}
\end{equation}
Thus $\overline Z^k$ is the optional version of a closed
$\mathbb G^k$-martingale.  The stopped PRP in
\Cref{lem:reference-stopped-prp} supplies a continuous martingale version
$Z^k$ and a $\mathbb G^k$-predictable, locally square-integrable process
$\beta^k$, chosen to vanish after $\tau_k$, such that
\begin{equation}
 Z_t^k=1+\int_0^{t\wedge\tau_k}(\beta_s^k)^\top\dd Y_s.
 \label{eq:mass-prp-beta}
\end{equation}

This continuous version is strictly positive simultaneously for all times.
Indeed, let
\[
 \zeta:=\inf\{t\in[0,T]:Z_t^k=0\}.
\]
On $\{\zeta<T\}$, optional sampling in
\eqref{eq:mass-closed-projection} gives
\[
 \E^{\Q}\!\left[
   \one_{\{\zeta<T\}}M^k_{\tau_k}
 \right]
 =\E^{\Q}\!\left[
   \one_{\{\zeta<T\}}Z_\zeta^k
 \right]=0.
\]
Since $M^k_{\tau_k}>0$ a.s.,
$\Q(\zeta<T)=0$.  Similarly,
\[
 \E^{\Q}[\one_{\{Z_T^k=0\}}M^k_{\tau_k}]
 =\E^{\Q}[\one_{\{Z_T^k=0\}}Z_T^k]=0,
\]
so $Z_T^k>0$ a.s. as well.  On the union of these null events set
$Z_t^k=1$ for every $t$.  The common completion makes the resulting process
adapted; it remains an indistinguishable continuous martingale version and is
now positive on every path used below.

Both $\overline Z^k=\langle\mu^k,1\rangle$ and $Z^k$ are optional versions of
the same optional projection.  By the optional section theorem their
disagreement set is evanescent.  On that set replace
$\mu_t^k$ by $Z_t^k\delta_0$, leaving it unchanged elsewhere.  This is one
optional-kernel modification and it gives
$\langle\mu_t^k,1\rangle=Z_t^k$ simultaneously for all $t$, without altering
any conditional-expectation or predictable-pairing identity.  Formula
\eqref{eq:KS-formula} now defines a probability kernel simultaneously in
time.  The fixed-time conditional-law assertion follows from the ordinary
Bayes formula
\[
 \E^{\Pp}[\varphi(X_{t\wedge\tau_k})\mid\mathcal G_t^k]
 =\frac{
   \E^{\Q}[M^k_{t\wedge\tau_k}
      \varphi(X_{t\wedge\tau_k})\mid\mathcal G_t^k]
 }{
   \E^{\Q}[M^k_{t\wedge\tau_k}\mid\mathcal G_t^k]
 }.
\]
See \citet[Chapter~3]{Kallianpur1980},
\citet[Chapters~3--5]{BainCrisan2009}, or
\citet[Chapter~5]{Xiong2008}.
\end{proof}

We next identify the Brownian coefficient without presupposing stochastic
integrability of the merely $L^1$ predictable projection.  

\begin{lemma}[Predictable conditional coefficients]
Define, componentwise, the $\Q$-predictable projection
\begin{equation}
 c_s^k
 :={}^{p,\mathbb G^k,\Q}\!\left(
      M_s^kH_s\one_{\{s\le\tau_k\}}
    \right),
 \label{eq:c-stopped-definition}
\end{equation}
and choose its representative to be zero on $(\tau_k,T]$.
Then $c^k$ is locally square integrable and
\begin{equation}
 Z_t^k
 =1+\int_0^{t\wedge\tau_k}(c_s^k)^\top\dd Y_s,
 \qquad
 c_s^k
 =\one_{\{s\le\tau_k\}}\ip{\mu_s^k}{h(\cdot,u_s)}
 \quad \dd s\otimes\dd\Q\text{-a.e.}
 \label{eq:mass-SDE-direct}
\end{equation}
Here and below the random-kernel pairing on the right denotes the predictable
version fixed in \eqref{eq:predictable-pairing}.  
\end{lemma}

\begin{proof}

First,
\begin{equation}
 \E^{\Q}\int_0^{\tau_k}M_s^k|H_s|\dd s
 =\E^{\Pp}\int_0^{\tau_k}|H_s|\dd s<\infty
 \label{eq:MH-L1}
\end{equation}
by \eqref{eq:H2-under-P}; hence $c^k$ in
\eqref{eq:c-stopped-definition} exists in $L^1(\dd s\otimes\dd\Q)$.
Let $\eta$ be a bounded $\mathbb G^k$-predictable elementary
$\R^m$-valued process and put
\[
 K_t:=\int_0^{t\wedge\tau_k}\eta_s^\top\dd Y_s.
\]
For $n,\ell\in\N$, define
\[
 \begin{aligned}
 \lambda_n
 &:=\inf\left\{t:\int_0^{t\wedge\tau_k}
                  |\beta_s^k|^2\dd s\ge n\right\}\wedge T,\\
 \kappa_\ell&:=\inf\{t:|K_t|\ge\ell\}\wedge T,
 \qquad \sigma_{n,\ell}:=\lambda_n\wedge\kappa_\ell.
 \end{aligned}
\]
Then the $Z^k$-integral stopped at $\sigma_{n,\ell}$ is square-integrable,
$K^{\sigma_{n,\ell}}$ is bounded, and $K_{\sigma_{n,\ell}}$ is
$\mathcal G_{\sigma_{n,\ell}}^k$-measurable.  Moreover,
\[
 Z_{\sigma_{n,\ell}}^k
 =\E^{\Q}[M_{\sigma_{n,\ell}}^k
           \mid\mathcal G_{\sigma_{n,\ell}}^k].
\]
The stochastic-integral product identity, followed by the defining property
of predictable projection, therefore gives
\begin{align}
 \E^{\Q}\int_0^{\sigma_{n,\ell}}
      \eta_s^\top c_s^k\dd s
 &=\E^{\Q}[(Z_{\sigma_{n,\ell}}^k-1)
                 K_{\sigma_{n,\ell}}]
  =\E^{\Q}[(M_{\sigma_{n,\ell}}^k-1)
                 K_{\sigma_{n,\ell}}]\notag\\
 &=\E^{\Q}\int_0^{\sigma_{n,\ell}}
      \eta_s^\top M_s^kH_s\one_{\{s\le\tau_k\}}\dd s
  =\E^{\Q}\int_0^{\sigma_{n,\ell}}
      \eta_s^\top\beta_s^k\dd s.
 \label{eq:mass-coefficient-testing}
\end{align}
Here the middle product identity is valid because $M^k-1$ is an
$H^1(\Q)$ martingale by \eqref{eq:M-H1}, whereas the stopped $K$ is bounded;
equivalently it follows by one further square-integrable localization and
$H^1$ convergence.  For fixed $n$, let $\ell\uparrow\infty$ in
\eqref{eq:mass-coefficient-testing}.  The left side is uniformly integrable
by Cauchy--Schwarz on $[0,\lambda_n]$, and the two right-side integrands are
dominated in $L^1$ by \eqref{eq:MH-L1}.  Hence
\[
 \E^{\Q}\int_0^{\lambda_n}
   \eta_s^\top(\beta_s^k-c_s^k)\dd s=0.
\]
Since bounded predictable elementary processes determine
$\dd s\otimes\dd\Q$, a monotone-class argument gives
$\beta^k=c^k$ on $[0,\lambda_n]$.  Finally $\lambda_n\uparrow T$ almost
surely, and therefore
\[
 \beta_s^k=c_s^k
 \quad \dd s\otimes\dd\Q\text{-a.e.}
\]
Thus $c^k$ is locally square integrable by
\eqref{eq:mass-prp-beta}, and the first identity in
\eqref{eq:mass-SDE-direct} is established noncircularly.  The second identity
in \eqref{eq:mass-SDE-direct} follows from the predictable-simple-kernel
extension in \eqref{eq:predictable-pairing} and
\eqref{eq:stopped-control-factorization}.
\end{proof}
\begin{lemma}
Set
\begin{equation}
 m_s^k
 :=\begin{cases}
     c_s^k/Z_s^k,&s\le\tau_k,\\
     0,&s>\tau_k.
   \end{cases}
 \label{eq:m-stopped-definition}
\end{equation}
Then $m^k$ is $\mathbb G^k$-predictable and locally square integrable, and
only on the stopped interval one has
\begin{equation}
 m_s^k
 =\one_{\{s\le\tau_k\}}\ip{\pi_s^k}{h(\cdot,u_s)}
 \quad \dd s\otimes\dd\Q\text{-a.e.}
 \label{eq:m-predictable}
\end{equation}
Moreover,
\begin{equation}
 Z_t^k
 =\mathcal E\!\left(
    \int_0^{\,\cdot\wedge\tau_k}(m_s^k)^\top\dd Y_s
   \right)_t.
 \label{eq:mass-exponential}
\end{equation}
The process $m^k$ is also the
$(\mathbb G^k,\Pp)$-predictable projection of
$H\one_{[0,\tau_k]}$, and
\begin{equation}
 \E^{\Pp}\int_0^{\tau_k}|m_s^k|^2\dd s
 =\E^{\Q}\int_0^{\tau_k}Z_s^k|m_s^k|^2\dd s
 <\infty.
 \label{eq:m-P-L2}
\end{equation}

\end{lemma}

\begin{proof}
Predictable conditional Cauchy--Schwarz, applied componentwise, gives
\begin{equation*}
 Z_s^k|m_s^k|^2
 =\frac{|c_s^k|^2}{Z_s^k}
 \le{}^{p,\mathbb G^k,\Q}\!\left(
       M_s^k|H_s|^2\one_{\{s\le\tau_k\}}
     \right)
 \quad \dd s\otimes\dd\Q\text{-a.e.}
 \label{eq:mass-jensen}
\end{equation*}
Consequently,
\begin{equation}
 \E^{\Q}\int_0^{\tau_k}Z_s^k|m_s^k|^2\dd s
 \le \E^{\Q}\int_0^{\tau_k}M_s^k|H_s|^2\dd s
 =\E^{\Pp}\int_0^{\tau_k}|H_s|^2\dd s<\infty.
 \label{eq:m-weighted-L2}
\end{equation}
Stopping when $Z^k$ first falls below $1/n$, together with a local
square-integrability sequence for $c^k$, proves local square integrability of
$m^k$.  Since $Z^k$ is continuous and strictly positive on the compact time
interval, these stopping times increase to $T$.  Equations
\eqref{eq:mass-SDE-direct} and \eqref{eq:m-stopped-definition} then give
\eqref{eq:mass-exponential}, while division by $Z^k$ and
\eqref{eq:predictable-pairing} give the stopped identity
\eqref{eq:m-predictable}.

  For every bounded scalar $\mathbb G^k$-predictable process
$\alpha$ (and componentwise for vector tests), Fubini and the density-process
property give
\begin{align}
 &\E^{\Pp}\int_0^T
    \alpha_s H_s\one_{\{s\le\tau_k\}}\dd s\notag\\
 &\quad=\E^{\Q}\int_0^T
    \alpha_s M_s^kH_s\one_{\{s\le\tau_k\}}\dd s
  =\E^{\Q}\int_0^T\alpha_s c_s^k\dd s\notag\\
 &\quad=\E^{\Q}\int_0^T\alpha_s Z_s^km_s^k\dd s
  =\E^{\Pp}\int_0^T\alpha_s m_s^k\dd s.
 \label{eq:predictable-Bayes}
\end{align}
The last equality uses
$Z_s^k=\E^{\Q}[M_s^k\mid\mathcal G_s^k]$.
Thus $m^k={}^{p,\mathbb G^k,\Pp}
(H\one_{[0,\tau_k]})$.  Applying the same density calculation first to
$|m_s^k|^2\wedge n$ and then using monotone convergence and
\eqref{eq:m-weighted-L2} gives
\[
 \E^{\Pp}\int_0^{\tau_k}|m_s^k|^2\dd s
 =\E^{\Q}\int_0^{\tau_k}M_s^k|m_s^k|^2\dd s
 =\E^{\Q}\int_0^{\tau_k}Z_s^k|m_s^k|^2\dd s<\infty,
\]
which is \eqref{eq:m-P-L2}.
\end{proof}

\begin{lemma}
For every $\varphi\in C_c^2(\R^d)$, define the stopped predictable
version
\begin{equation*}
 B_s^k(\varphi)
 :=\one_{\{s\le\tau_k\}}
   \left(
    \ip{\pi_s^k}{\varphi h(\cdot,u_s)}
    -\ip{\pi_s^k}{\varphi}\,m_s^k
   \right).
 \label{eq:m-B}
\end{equation*}
It is $\mathbb G^k$-predictable and locally square integrable.
\end{lemma}

\begin{proof}

If $\varphi\in C_c^2(\R^d)$, the Zakai equation makes
$\langle\mu^k,\varphi\rangle$ continuous.  Hence
$\langle\pi^k,\varphi\rangle
=\langle\mu^k,\varphi\rangle/Z^k$ has a continuous
$\mathbb G^k$-predictable version.  We define the stopped predictable version
of $\langle\pi_s^k,\varphi h(\cdot,u_s)\rangle$ as
$(Z_s^k)^{-1}$ times the corresponding predictable $\mu^k$-pairing from
\eqref{eq:predictable-pairing}.  This proves the predictability asserted for
$B^k(\varphi)$.  The fact that $\pi^k$ is a probability kernel also gives, on
$[0,\tau_k]$,
\[
 \left|\ip{\pi_s^k}{\varphi h(\cdot,u_s)}\right|
 \le C_{\varphi,K}(1+|u_s|),
 \qquad
 \left|\ip{\pi_s^k}{\varphi}\right|
 \le\|\varphi\|_\infty.
\]
It follows that
\[
 |B_s^k(\varphi)|^2
 \le C_{\varphi,K}
      (1+|u_s|^2+|m_s^k|^2)\one_{\{s\le\tau_k\}}.
\]
The pathwise control-energy bound and local square integrability of $m^k$
prove the last assertion.
\end{proof}

\begin{theorem}[Stopped Kushner--Stratonovich equation in common-observation form]
\label{thm:stopped-KS}
For every $\varphi\in C_c^2(\R^d)$,
\begin{align}
 \ip{\pi_t^k}{\varphi}
 ={}&\ip{\nu_0}{\varphi}
 +\int_0^{t\wedge\tau_k}\ip{\pi_s^k}{L_s\varphi}\dd s
 \notag\\
 &+\int_0^{t\wedge\tau_k}B_s^k(\varphi)^\top\dd Y_s
 -\int_0^{t\wedge\tau_k}
       B_s^k(\varphi)^\top m_s^k\dd s.
 \label{eq:KS-common-Y}
\end{align}
Under $\Pp$, the process
\begin{equation*}
 I_t^k
 :=Y_{t\wedge\tau_k}
   -\int_0^{t\wedge\tau_k}m_s^k\dd s
 \label{eq:stopped-innovation}
\end{equation*}
is a continuous square-integrable $(\mathbb G^k,\Pp)$-martingale with
\begin{equation}
 [I^k]_t=(t\wedge\tau_k)I_m.
 \label{eq:innovation-bracket}
\end{equation}
Thus $I^k$ is the innovation Brownian motion on the stochastic interval
$[0,\tau_k]$, held constant after $\tau_k$.  Equivalently, the last two terms in
\eqref{eq:KS-common-Y} equal
$\int_0^{t\wedge\tau_k}B_s^k(\varphi)^\top\dd I_s^k$.
\end{theorem}

\begin{proof}
Write $N_t^\varphi=\ip{\mu_t^k}{\varphi}$.  By
\Cref{thm:stopped-zakai,lem:mass-process}, under $\Q$,
\begin{align*}
 \dd N_t^\varphi
 &=\ip{\mu_t^k}{L_t\varphi}\dd t
   +\ip{\mu_t^k}{\varphi h(\cdot,u_t)}^\top\dd Y_t,\\
 \dd Z_t^k&=Z_t^k(m_t^k)^\top\dd Y_t,
 \qquad t\le\tau_k.
\end{align*}
The first process is continuous by the Zakai equation; $Z^k$ is continuous and
strictly positive by \Cref{lem:mass-process}.  Moreover,
$m^k$ and $B^k(\varphi)$ are locally square-integrable, while
\[
 \left|\ip{\pi_s^k}{L_s\varphi}\right|
 \le C_{\varphi,K}(1+|u_s|^2),
\]
so every term in the quotient calculation is locally integrable.  Apply
It\^o's formula to $N^\varphi/Z^k$ after stopping when
$Z^k+(Z^k)^{-1}$ or
$\int_0^{\,\cdot\wedge\tau_k}(|m_s^k|^2+|B_s^k(\varphi)|^2)\dd s$
reaches $n$.  The stochastic coefficient is
\[
 \frac{\ip{\mu_t^k}{\varphi h(\cdot,u_t)}}{Z_t^k}
 -\frac{N_t^\varphi}{Z_t^k}m_t^k
 =B_t^k(\varphi),
\]
and the quadratic-covariation drift is
\[
 \frac{N_t^\varphi}{Z_t^k}|m_t^k|^2
 -\frac{\ip{\mu_t^k}{\varphi h(\cdot,u_t)}}{Z_t^k}^{\!\top}m_t^k
 =-B_t^k(\varphi)^\top m_t^k.
\]
Letting $n\uparrow\infty$ proves \eqref{eq:KS-common-Y} under $\Q$.  
Because $\Q^k\sim\Pp$ and stochastic integrals with respect to a fixed
continuous semimartingale are invariant under an equivalent change of
probability, \eqref{eq:KS-common-Y} is the same pathwise semimartingale
identity under $\Pp$.

It remains to verify the innovation assertion with the stopped versions just
fixed.  Equation \eqref{eq:predictable-Bayes} says exactly that
\[
 m^k
 ={}^{p,\mathbb G^k,\Pp}
    (H\one_{[0,\tau_k]}),
\]
and \eqref{eq:m-P-L2} supplies the required square integrability.  More
explicitly, if $\alpha$ is a bounded $\mathbb G^k$-predictable elementary
$\R^m$-valued process, then, using
$\dd Y_s=H_s\dd s+\dd V_s$ under $\Pp$ and the fact that $\alpha$ is also
$\mathbb F$-predictable,
\begin{align*}
 \E^{\Pp}\int_0^{\tau_k}\alpha_s^\top\dd Y_s
 &=\E^{\Pp}\int_0^{\tau_k}\alpha_s^\top H_s\dd s\\
 &=\E^{\Pp}\int_0^{\tau_k}\alpha_s^\top m_s^k\dd s.
\end{align*}
Taking $\alpha_s=\xi\one_{(r,t]}(s)e_j$ with bounded
$\xi\in\mathcal G_r^k$ gives directly
$\E^{\Pp}[\xi(I_t^{k,j}-I_r^{k,j})]=0$; the required integrability follows
from \eqref{eq:H2-under-P} and \eqref{eq:m-P-L2}.  A monotone-class argument
therefore shows that
\[
 I_t^k=Y_{t\wedge\tau_k}
       -\int_0^{t\wedge\tau_k}m_s^k\dd s
\]
is a continuous $(\mathbb G^k,\Pp)$-martingale.  Its quadratic variation is
pathwise
\[
 [I^k]_t=[Y_{\cdot\wedge\tau_k}]_t
 =(t\wedge\tau_k)I_m.
\]
The bracket is bounded by $T I_m$, so $I^k$ is square-integrable and
\eqref{eq:innovation-bracket} follows.  Substitution of
$\dd I_t^k=\dd Y_t-m_t^k\dd t$ on $[0,\tau_k]$ gives the final formulation.

\end{proof}

\section{Pathwise uniqueness for stopped Zakai equation}
\label{sec:measure-duality}

The aim of this section is to establish the pathwise uniqueness of the stopped Zakai equation 
\eqref{eq:zakai-stopped} by duality argument.

\subsection{Admissible measure-valued solutions}

Let $\cM_+(\R^d)$ denote the finite nonnegative Borel measures on $\R^d$, equipped with the narrow topology. Fix $k$ and 
take $q_0=d+4$.

\begin{definition}[Admissible stopped Zakai solution]
\label{def:admissible-zakai}
An $\mathbb F^Y$-adapted process
$\mu:[0,\tau_k]\times\Omega\to\cM_+(\R^d)$ is an admissible solution of the stopped Zakai equation with initial measure $\nu$ if:
\begin{enumerate}[label=\textnormal{(\roman*)}]
\item its paths are narrowly continuous;
\item for every $\varphi\in C_c^2(\R^d)$, it satisfies \eqref{eq:zakai-stopped} with $\nu$ in place of $\nu_0$;
\item the family
\begin{equation*}
 \left\{\ip{\mu_\rho}{1}:\rho\le\tau_k
     \text{ is an }\mathbb F^Y\text{-stopping time}\right\}
 \label{eq:mass-class-D}
\end{equation*}
is uniformly integrable;
\item there are stopping times $\eta_N\uparrow\tau_k$ such that
\begin{equation}
 \sup_{0\le s\le\eta_N}
 \ip{\mu_s}{\br{\cdot}^{q_0}}<\infty
 \quad\Q\text{-a.s. for every }N.
 \label{eq:local-q0-moment}
\end{equation}
\end{enumerate}
\end{definition}

\begin{proposition}[The actual unnormalized filter is admissible]
\label{prop:actual-admissible}
Assume, in addition to \Cref{ass:filtering}, that
\begin{equation}
 \int_{\R^d}\br{x}^{q_0}\nu_0(\dd x)<\infty.
 \label{eq:initial-q0}
\end{equation}
Then the unnormalized filter $\mu^k$ defined by \eqref{eq:mu-definition} is an admissible stopped Zakai solution.
\end{proposition}

\begin{proof}
We first construct the narrowly continuous version.  Choose a countable set
$\{\varphi_j:j\ge1\}\subset C_c^2(\R^d)$ that is convergence determining for finite measures and includes a sequence of cutoffs increasing to $1$.  For every $j$, the Zakai equation supplies a continuous version of
$t\mapsto\langle\mu_t^k,\varphi_j\rangle$.  Intersect the corresponding countably many full-probability events.

Let
\[
 \mathcal X:=\sup_{0\le s\le\tau_k}\br{X_s}^{q_0}.
\]
The stopped SDE moment estimate under $\Pp$ and \eqref{eq:initial-q0} give
$\E^{\Pp}\mathcal X<\infty$.  Since $M_{\tau_k}^k$ is the density of $\Pp$ with respect to $\Q$ on $\F_{\tau_k}$, define the finite closed $\mathbb F^Y$-martingale
\[
 \mathcal Z_t:=\E^\Q[M_{\tau_k}^k\mathcal X\mid\FY_{t\wedge\tau_k}].
\]

Let $w(x):=\br{x}^{q_0}$ and let ${}^{o,Y}$ denote optional projection
onto the stopped observation filtration
$\bigl(\FY_{t\wedge\tau_k}\bigr)_{0\le t\le T}$.  Applied to the
nonnegative measurable processes
\[
 U_t:=M_{\tau_k}^k w(X_{t\wedge\tau_k}),
 \qquad
 \overline U_t:=M_{\tau_k}^k\mathcal X,
\]
the optional projections exist because $0\le U\le\overline U$ and
$\E^\Q[\overline U_t]=\E^\Pp[\mathcal X]<\infty$.  The defining
conditional-expectation identity for $\mu^k$, first applied to $w\wedge r$
and then followed by monotone convergence, shows that for every
$\mathbb F^Y$-stopping time $\sigma\le\tau_k$,
\[
 {}^{o,Y}U_\sigma=\langle\mu_\sigma^k,w\rangle
 \quad\Q\text{-a.s.}
\]
We choose this optional version of the $w$-pairing.  Similarly,
${}^{o,Y}\overline U=\mathcal Z$.  Monotonicity of optional projection,
followed by the optional section theorem, therefore gives, outside one
evanescent set,
\begin{equation}
 \langle\mu_t^k,w\rangle
 ={}^{o,Y}U_t
 \le{}^{o,Y}\overline U_t
 =\mathcal Z_t,
 \qquad 0\le t\le T.
 \label{eq:actual-filter-moment-bound}
\end{equation}
In particular,
\[
 \sup_{0\le t\le T}\langle\mu_t^k,w\rangle
 \le \sup_{0\le t\le T}\mathcal Z_t<\infty
 \qquad\Q\text{-a.s.},
\]
because every c\`adl\`ag path of the finite closed martingale $\mathcal Z$
is bounded on the compact time interval.  Hence the family
$\{\mu_t^k:0\le t\le\tau_k\}$ is pathwise tight.

If $t_n\to t$, continuity of all pairings with $\varphi_j$ and tightness show that every subsequence of $\mu_{t_n}^k$ has a narrowly convergent subsubsequence; the convergence-determining family identifies its only possible limit as $\mu_t^k$.  Thus the selected version is narrowly continuous.

Because $M^k$ is a closed $\Q$-martingale,
\begin{equation*}
 \ip{\mu_t^k}{1}
 =\E^\Q[M_{\tau_k}^k\mid\FY_{t\wedge\tau_k}].
 \label{eq:mass-closed}
\end{equation*}
The mass process is therefore a closed nonnegative martingale and is of class $D$.  Finally, \eqref{eq:actual-filter-moment-bound} gives the  $q_0$-moment bound required in \eqref{eq:local-q0-moment}.  Hence all four admissibility conditions hold.
\end{proof}

\subsection{A finite-rank stochastic product rule}

We first isolate the elementary product formula from which the measure-valued identity is obtained. Recall that the weighted Sobolev space
$H^n_\ell(K_1;K_2)$ is defined in \cite{XXY} and in the appendix.

\begin{lemma}[Finite-rank random test fields]
\label{lem:finite-rank-product}
Let $\rho\le\tau_k$ be a stopping time and suppose first that
\begin{equation*}
 \psi_t(x)=\sum_{j=1}^N\zeta_t^j\phi_j(x),
 \qquad \phi_j\in C_c^2(\R^d),
 \label{eq:finite-rank-field}
\end{equation*}
where
\begin{equation*}
 \dd\zeta_t^j=A_t^j\dd t+\sum_{a=1}^mB_t^{j,a}\dd Y_t^a.
 \label{eq:zeta-semimart}
\end{equation*}
Put $A_t=\sum_jA_t^j\phi_j$ and $B_t^a=\sum_jB_t^{j,a}\phi_j$.  If all terms are integrable after stopping at $\rho$, then
\begin{align}
 \ip{\mu_{t\wedge\rho}}{\psi_{t\wedge\rho}}
 ={}&\ip{\mu_0}{\psi_0}
 +\int_0^{t\wedge\rho}
    \ip{\mu_s}{L_s\psi_s+A_s+h(\cdot,u_s)^\top B_s}\dd s
 \notag\\
 &+\int_0^{t\wedge\rho}
    \ip{\mu_s}{B_s+\psi_sh(\cdot,u_s)}^\top\dd Y_s.
 \label{eq:finite-rank-product-formula}
\end{align}

The same formula holds in the following precise closure class.  Let $K\Subset\R^d$ be a bounded smooth domain and choose $s_0>d/2+2$.  Assume that all fields are supported in $K$ and
\begin{align}
 &\psi\in L^2\bigl(\Omega;C([0,\rho];H_0^{s_0}(K))\bigr),\notag\\
 &A\in L^2\bigl(\Omega;L^1(0,\rho;H_0^{s_0}(K))\bigr),\qquad
 B\in L^2_{\mathcal P}\bigl(\Omega\times(0,\rho);H_0^{s_0}(K;\R^m)\bigr),
 \label{eq:Sobolev-field-hypotheses}
\end{align}
and that the linked identity
\begin{equation}
 \psi_t=\psi_0+\int_0^tA_s\dd s+\int_0^tB_s^\top\dd Y_s
 \quad\text{holds in }H_0^{s_0}(K).
 \label{eq:linked-Sobolev-decomp}
\end{equation}
Then \eqref{eq:finite-rank-product-formula} holds with these $\psi,A,B$ after the natural coefficient and measure localizations.
\end{lemma}

\begin{proof}
For a finite-rank field, apply scalar integration by parts to
$\zeta_t^j\langle\mu_t,\phi_j\rangle$ and sum over $j$.  The quadratic covariation between
$\sum_aB_t^{j,a}\dd Y_t^a$ and
$\sum_a\langle\mu_t,\phi_jh^a\rangle\dd Y_t^a$ equals
$\sum_aB_t^{j,a}\langle\mu_t,\phi_jh^a\rangle\dd t$.  Collecting the drift and stochastic terms gives \eqref{eq:finite-rank-product-formula}.

For the extension, take a countable dense subset of $C_c^\infty(K)$ in $H_0^{s_0}(K)$ and apply Gram--Schmidt.  The resulting orthonormal basis $(e_j)_{j\ge1}$ still consists of functions in $C_c^\infty(K)$.  Let $P_n$ be the orthogonal projection onto $\operatorname{span}\{e_1,\ldots,e_n\}$.  Taking $H_0^{s_0}$ inner products in \eqref{eq:linked-Sobolev-decomp} shows that
\[
 \zeta_t^j:=(\psi_t,e_j)_{H^{s_0}}
\]
satisfies
\[
 \dd\zeta_t^j=(A_t,e_j)_{H^{s_0}}\dd t
 +\sum_{a=1}^m(B_t^a,e_j)_{H^{s_0}}\dd Y_t^a.
\]
Thus $P_n\psi,P_nA,P_nB$ preserve the linked semimartingale decomposition.

For almost every $\omega$, the range of the continuous path $t\mapsto\psi_t(\omega)$ is compact in $H_0^{s_0}(K)$, and $P_n$ converges uniformly on compact subsets.  Hence
\begin{align*}
 P_n\psi&\to\psi &&\text{in }L^2(\Omega;C_tH^{s_0}),\\
 P_nA&\to A &&\text{in }L^2(\Omega;L_t^1H^{s_0}),\\
 P_nB&\to B &&\text{in }L^2_{\mathcal P}(\Omega\times(0,\rho);H^{s_0}),
\end{align*}
where $C_tH^{s_0}=C([0,\rho],H^{s_0}_0(K))$ and $L_t^1H^{s_0}=L^1(0,\rho,H^{s_0}_0(K))$.
Predictability is preserved by deterministic projection.  Since $s_0>d/2+2$, Sobolev embedding converts these convergences to the $C_x^2$, $L_t^1C_x^2$, and $L_t^2C_x^2$ convergences needed on $K$.

Write
\[
 \delta\psi^n:=P_n\psi-\psi,
 \qquad \delta A^n:=P_nA-A,
 \qquad \delta B^n:=P_nB-B.
\]
For $\ell\in\N$, put $c_\ell:=\ell+\langle\mu_0,1\rangle$ and define
the explicit common localization
\begin{equation*}
 \vartheta_\ell
 :=\rho\wedge\inf\left\{t\ge0:
     \langle\mu_t,1\rangle
     +\int_0^{t\wedge\rho}(1+|u_s|^2)\dd s
     \ge c_\ell\right\}.
 \label{eq:finite-rank-closure-localizer}
\end{equation*}
Narrow continuity of $\mu$, continuity of the time integral, and the stopped
control-energy bound imply $\vartheta_\ell\uparrow\rho$ a.s..  On
$[0,\vartheta_\ell]$ the mass is bounded by $c_\ell$, and, since all fields
are supported in $K$,
\[
 \sup_{x\in K}\bigl(|b(x,u_s)|+|a(x,u_s)|+|h(x,u_s)|^2\bigr)
 \le C_K(1+|u_s|^2).
\]
Sobolev embedding and the three convergences above give
\begin{equation}
 \E^\Q\sup_{t\le\vartheta_\ell}
   \left|\langle\mu_t,\delta\psi_t^n\rangle\right|^2
 \le C_{K,\ell}\E^\Q\sup_{t\le\rho}
       \|\delta\psi_t^n\|_{H^{s_0}(K)}^2
 \longrightarrow0,
 \label{eq:finite-rank-closure-boundary}
 \end{equation}
 \begin{equation}
 \E^\Q\int_0^{\vartheta_\ell}
  \left|\left\langle\mu_s,
       L_s\delta\psi_s^n+\delta A_s^n+h(\cdot,u_s)^\top\delta B_s^n
       \right\rangle\right|\dd s
 \longrightarrow0,
 \label{eq:finite-rank-closure-drift}
 \end{equation}
 and
 \begin{eqnarray}
 \E^\Q\int_0^{\vartheta_\ell}
  \left|\left\langle\mu_s,
       \delta B_s^n+\delta\psi_s^n h(\cdot,u_s)
       \right\rangle\right|^2\dd s
 &\le& C_{K,\ell}
     \E^\Q\int_0^\rho\|\delta B_s^n\|_{H^{s_0}(K)}^2\dd s\notag\\
 &&+C_{K,\ell}
     \E^\Q\sup_{s\le\rho}\|\delta\psi_s^n\|_{H^{s_0}(K)}^2
 \longrightarrow0.
 \label{eq:finite-rank-closure-martingale}
\end{eqnarray}
For \eqref{eq:finite-rank-closure-drift}, use the
$L^2(\Omega;L_t^1H^{s_0})$ convergence for $\delta A^n$, the
uniform-in-time convergence for $\delta\psi^n$, and Cauchy--Schwarz in time
for the term containing $\delta B^n$.  The endpoint estimate
\eqref{eq:finite-rank-closure-boundary} also applies at
$t\wedge\vartheta_\ell$, including the crossing time, because the mass
process is continuous.

Apply the finite-rank identity to $(P_n\psi,P_nA,P_nB)$ stopped at
$\vartheta_\ell$.  Equations
\eqref{eq:finite-rank-closure-boundary}--
\eqref{eq:finite-rank-closure-martingale}, together with It\^o's isometry,
allow $n\to\infty$ and give \eqref{eq:finite-rank-product-formula} stopped
at $\vartheta_\ell$.  Since $\vartheta_\ell\uparrow\rho$, these identities
are consistent and prove the asserted localized formula on $[0,\rho]$.  If
the terms in \eqref{eq:finite-rank-product-formula} are globally integrable
up to $\rho$, dominated convergence for the drift and BDG for the stochastic
term also permit $\ell\to\infty$ directly.

\end{proof}

\subsection{The measure-valued It\^o product formula}

Let $(f,g)$ solve the BSPDE \eqref{eq:main-bspde} on an $\mathbb F^Y$-stopping time
$\theta\le\tau_k$.  For later estimates, set
\begin{eqnarray*}
 K_s^f&:=&\max_{|\alpha|\le2}
       \sup_{x\in\R^d}\frac{|D^\alpha f_s(x)|}{\br{x}^{\lambda}},
 \label{eq:Kf}\\
 K_s^g&:=&\sup_{x\in\R^d}
       \frac{|g_s(x)|}{\br{x}^{\lambda}}.
 \label{eq:Kg}
\end{eqnarray*}
By \Cref{thm:main-bspde},
\begin{equation}
 \E^\Q\esssup_{s\le\theta}|K_s^f|^2
 +\E^\Q\int_0^\theta|K_s^g|^2\dd s<\infty.
 \label{eq:Kfg-integrability}
\end{equation}

\begin{lemma}[Convergence of random terminal pairings]
\label{lem:terminal-pairing-convergence}
Let $\mu$ be an admissible Zakai solution and let $\sigma_j\uparrow\theta$ be stopping times.  Then
\begin{equation}
 \ip{\mu_{\sigma_j}}{f_{\sigma_j}}
 \longrightarrow \ip{\mu_\theta}{\gamma}
 \quad\text{in probability}.
 \label{eq:terminal-pairing-convergence}
\end{equation}
If the mass process is of class $D$, the family on the left is uniformly integrable and the convergence is in $L^1$.
\end{lemma}

\begin{proof}

Work on a common full-probability event on which the paths of $\mu$
are narrowly continuous.
Then $\mu_{\sigma_j}\to\mu_\theta$ narrowly and
$f_{\sigma_j}\to f_\theta=\gamma$ uniformly on every compact subset
of $\R^d$.  In particular,
$\{\mu_{\sigma_j}:j\ge1\}\cup\{\mu_\theta\}$ has uniformly bounded
masses and is uniformly tight.

Put $C_\gamma:=\norm{\gamma}_{L^\infty(\Omega\times\R^d)}$.  For a
fixed sample point and $\varepsilon>0$, choose $R$ such that
\[
 \sup_{j\ge1}\mu_{\sigma_j}(B_R^c)+\mu_\theta(B_R^c)<\varepsilon.
\]
Since $|f_s|\le C_\gamma$ by \eqref{eq:f-bounded},
\begin{align*}
 \left|\ip{\mu_{\sigma_j}}{f_{\sigma_j}}
          -\ip{\mu_\theta}{\gamma}\right|
 &\le \mu_{\sigma_j}(B_R)
       \sup_{x\in B_R}|f_{\sigma_j}(x)-\gamma(x)|\\
 &\quad+2C_\gamma\mu_{\sigma_j}(B_R^c)
      +\left|\ip{\mu_{\sigma_j}-\mu_\theta}{\gamma}\right|.
\end{align*}
The first and third terms tend to zero because
$\gamma\in C_b(\R^d)$, and the limsup of the second is at most
$2C_\gamma\varepsilon$.  Letting $\varepsilon\downarrow0$ proves
\eqref{eq:terminal-pairing-convergence} almost surely, without using
a $q_0$-moment at the endpoint.

Finally,
$|\ip{\mu_{\sigma_j}}{f_{\sigma_j}}|
\le C_\gamma\ip{\mu_{\sigma_j}}1$.
The mass class-$D$ property gives uniform integrability; the same
bound at $\theta$ makes the limit integrable.  Vitali's theorem
therefore gives convergence in $L^1$.

\end{proof}

Since we are going to use the solution of the BSPDE stated in the Appendix, we  shall 
make the strngthen version \Cref{ass:bspde-regularity} of \Cref{ass:filtering}
as a standing hypothesis in this and the next sections.

\begin{theorem}[Measure-valued stochastic product formula]
\label{thm:measure-product}
Suppose \Cref{ass:bspde-regularity} holds.
Let $\mu$ be an admissible Zakai solution on $[0,\theta]$ and let $(f,g)$ solve \eqref{eq:main-bspde} with terminal field
$\gamma\in\mathscr G_\theta^{n_*+1}$.  Then, for every $t\in[0,T]$,
\begin{equation}
 \ip{\mu_{t\wedge\theta}}{f_{t\wedge\theta}}
 =\ip{\mu_0}{f_0}
 +\int_0^{t\wedge\theta}
   \ip{\mu_s}{g_s+f_sh(\cdot,u_s)}^\top\dd Y_s.
 \label{eq:measure-product}
\end{equation}
Moreover,
\begin{equation}
 \E^\Q\ip{\mu_\theta}{\gamma}=\E^\Q\ip{\mu_0}{f_0}.
 \label{eq:duality-expectation}
\end{equation}
\end{theorem}

\begin{proof}

Let $w(x):=\br{x}^{q_0}$ and choose the increasing stopping times
$\eta_n\uparrow\theta$ from \eqref{eq:local-q0-moment}.  Extend
$\mu$ constantly after $\theta$ and, for $r\in\N$, set
\[
 Q_t^{(r)}:=\ip{\mu_{t\wedge\theta}}{w\wedge r},
 \qquad
 Q_t:=\sup_{r\in\N}Q_t^{(r)}
       =\ip{\mu_{t\wedge\theta}}w.
\]
Each $Q^{(r)}$ is continuous and adapted by narrow continuity; hence
$Q$ is progressively measurable.  Since the initial measure is fixed,
$Q_0<\infty$ is deterministic.  Put
\begin{equation*}
 c_n:=n+Q_0,\qquad
 \beta_n:=\inf\{t\in[0,\theta]:Q_t>c_n\}\wedge\theta,
 \qquad
 \alpha_n:=\eta_n\wedge\beta_n.
 \label{eq:product-localizers}
\end{equation*}
The debut theorem for progressive sets shows that $\beta_n$ and
$\alpha_n$ are stopping times.  The sequence is increasing and
$\alpha_n\uparrow\theta$: for every $t<\theta$, admissibility bounds
$\sup_{s\le t}Q_s$ after some $\eta_N$, so eventually
$t<\alpha_n$.  Moreover,
\begin{equation}
 Q_s\le c_n\qquad\text{for every }s<\alpha_n.
 \label{eq:product-moment-before-stop}
\end{equation}
The endpoint $Q_{\alpha_n}$ need not be bounded.  The mass process
$Z_t^\mu:=\ip{\mu_{t\wedge\theta}}1$, however, is continuous and
$Z_s^\mu\le Q_s$.  Taking a limit from the left gives
\begin{equation}
 \ip{\mu_{\alpha_n}}1\le c_n.
 \label{eq:product-mass-at-stop}
\end{equation}
Thus the $q_0$-moment is used only under time integrals, while
random-time boundary pairings are controlled by the mass.

Let $\rho_\varepsilon$ be a standard nonnegative spatial mollifier
and write $f^\varepsilon=f*\rho_\varepsilon$ and
$g^\varepsilon=g*\rho_\varepsilon$.  Choose
$\chi\in C_c^\infty$ equal to one on $B_1$ and zero outside $B_2$,
and put $\chi_R(x)=\chi(x/R)$.  For fixed $\varepsilon,R$, define
\[
 \psi^{\varepsilon,R}:=\chi_Rf^\varepsilon,
 \qquad
 A^{\varepsilon,R}:=-\chi_R\bigl[(L_sf_s)^\varepsilon
                         +(h_s^\top g_s)^\varepsilon\bigr],
\]
\[
 B^{\varepsilon,R}:=\chi_Rg^\varepsilon,
 \qquad h_s(x):=h(x,u_s).
\]
Choose a bounded smooth domain $K_R$ whose interior contains
$\supp\chi_R$.  For every fixed $s_0>d/2+2$, spatial smoothing,
\Cref{thm:main-bspde}, the
coefficient bounds, and the control-energy bound give
\begin{align*}
 &\E^\Q\sup_{s\le\theta}
       \|\psi_s^{\varepsilon,R}\|_{H_0^{s_0}(K_R)}^2
 +\E^\Q\left(\int_0^\theta
       \|A_s^{\varepsilon,R}\|_{H_0^{s_0}(K_R)}\dd s\right)^2\\
 &\qquad
 +\E^\Q\int_0^\theta
       \|B_s^{\varepsilon,R}\|_{H_0^{s_0}(K_R)}^2\dd s<\infty.
\end{align*}

We justify convolution of the stochastic term at the Hilbert-space level.
Without loss of generality take $0<\varepsilon\le1$, and define the
deterministic smoothing operator
\[
 T_{\varepsilon,R}v:=\chi_R(v*\rho_\varepsilon).
\]
For fixed $(\varepsilon,R,s_0)$, convolution followed by multiplication by
$\chi_R$ is a bounded linear map from $H^{-2}(B_{2R+1})$ into
$H_0^{s_0}(K_R)$; on nonnegative Sobolev orders the same statement follows
a fortiori.  In particular, the estimates above imply
\begin{align*}
 &\E^\Q\int_0^\theta
   \|T_{\varepsilon,R}g_s\|_{H^{s_0}(K_R;\R^m)}^2\dd s<\infty,\\
 &\E^\Q\left(\int_0^\theta
   \|T_{\varepsilon,R}(L_sf_s+h_s^\top g_s)\|_{H^{s_0}(K_R)}\dd s\right)^2<\infty.
\end{align*}
Apply $T_{\varepsilon,R}$ to the distribution-valued weak BSPDE.  A
deterministic bounded linear operator commutes with Bochner integration.  It
also commutes with the Hilbert-valued It\^o integral: this is immediate for
elementary predictable integrands, and the general case follows from It\^o's
isometry and the first estimate above.  Consequently, outside one null set
and for every $t\in[0,T]$,
\begin{equation}
 \psi_{t\wedge\theta}^{\varepsilon,R}
 =\psi_0^{\varepsilon,R}
  +\int_0^{t\wedge\theta}A_s^{\varepsilon,R}\dd s
  +\int_0^{t\wedge\theta}(B_s^{\varepsilon,R})^\top\dd Y_s
 \quad\text{in }H_0^{s_0}(K_R).
 \label{eq:smoothed-linked-Hilbert-identity}
\end{equation}
Equivalently, on $[0,\theta]$,
\[
 \dd f_s^\varepsilon
 =-\bigl[(L_sf_s)^\varepsilon+(h_s^\top g_s)^\varepsilon\bigr]\dd s
 +(g_s^\varepsilon)^\top\dd Y_s
\]
after multiplication by $\chi_R$.  Thus
\eqref{eq:Sobolev-field-hypotheses}--
\eqref{eq:linked-Sobolev-decomp} hold, with the linked identity understood in
the precise Hilbert-valued sense of
\eqref{eq:smoothed-linked-Hilbert-identity}.

On $[0,\alpha_n]$,
\eqref{eq:product-moment-before-stop} controls every time-integrated
pairing and \eqref{eq:product-mass-at-stop} controls the boundary
pairing.  Lemma~\ref{lem:finite-rank-product} therefore applies and
yields the following identity:
\begin{eqnarray*}
 \ip{\mu_{t\wedge\alpha_n}}{\chi_Rf_{t\wedge\alpha_n}^\varepsilon}
 &={}&\ip{\mu_0}{\chi_Rf_0^\varepsilon}
 +\int_0^{t\wedge\alpha_n}
   \ip{\mu_s}{\chi_R(g_s^\varepsilon+f_s^\varepsilon h_s)}^\top\dd Y_s
 \notag\\
& &+\int_0^{t\wedge\alpha_n}
   \ip{\mu_s}{C_s^{\varepsilon,R}+\chi_Rr_s^\varepsilon+\chi_Rs_s^\varepsilon}\dd s,
 \label{eq:eps-R-product}
\end{eqnarray*}
where
\begin{align*}
 C_s^{\varepsilon,R}&:=L_s(\chi_Rf_s^\varepsilon)-\chi_RL_sf_s^\varepsilon,\\
 r_s^\varepsilon&:=L_sf_s^\varepsilon-(L_sf_s)^\varepsilon,\\
 s_s^\varepsilon&:=h_s^\top g_s^\varepsilon-(h_s^\top g_s)^\varepsilon.
\end{align*}

\medskip
\noindent\emph{Step 1: $\varepsilon\downarrow0$ at fixed $R,n$.}

By \Cref{prop:local-time-continuity}, the path
$s\mapsto f_s$ is continuous with values in
$H^{n_*-2}(B_{2R+1})$, and $n_*-2>d/2$.  Uniform convergence of an
approximate identity on the compact range of this path, followed by
Sobolev embedding, gives
\begin{equation}
 \sup_{s\le\theta}\sup_{x\in B_{2R}}
 |f_s^\varepsilon(x)-f_s(x)|\longrightarrow0
 \quad\Q\text{-a.s.}
 \label{eq:mollifier-C0-uniform-time}
\end{equation}
The difference is bounded by $2\norm{\gamma}_\infty$.  In particular,
for every fixed $t$,
\begin{equation*}
 \E^\Q\left|
 \ip{\mu_{t\wedge\alpha_n}}
     {\chi_R(f_{t\wedge\alpha_n}^\varepsilon
                  -f_{t\wedge\alpha_n})}
 \right|^2\longrightarrow0,
 \label{eq:mollifier-random-time-boundary}
\end{equation*}
by \eqref{eq:product-mass-at-stop}; the same holds at time zero.

For terms integrated against $\dd s$ or $\dd Y_s$,
\Cref{thm:main-bspde} and $n_*>d/2+2$ give
\begin{align}
 f^\varepsilon&\longrightarrow f
 &&\text{in }L^2(\Q\otimes\dd s;C^2(B_{2R})),
 \notag\\
 g^\varepsilon&\longrightarrow g
 &&\text{in }L^2(\Q\otimes\dd s;C^2(B_{2R};\R^m)).
 \label{eq:mollifier-local-time-integrated}
\end{align}
The coefficient
assumptions and commutator identities imply
\begin{eqnarray*}
 \sup_{x\in B_{2R}}|r_s^\varepsilon(x)|
 &\le& C_R\varepsilon(1+|u_s|)
   \sup_{x\in B_{2R+1}}(|Df_s(x)|+|D^2f_s(x)|),
 \label{eq:r-eps-bound}\\
 \sup_{x\in B_{2R}}|s_s^\varepsilon(x)|
 &\le& C\varepsilon\sup_{x\in B_{2R+1}}|g_s(x)|.
 \label{eq:s-eps-bound}
\end{eqnarray*}
For instance, the second-order commutator is the integral of
$[a^{ij}(x,u_s)-a^{ij}(x-y,u_s)]D_{ij}f_s(x-y)
\rho_\varepsilon(y)$, and on the fixed ball
$|D_xa|\le C_R(1+|u_s|)$.  By
\eqref{eq:product-moment-before-stop}, the pairings of the
right-hand sides are bounded by an integrable multiple of
$\varepsilon[(1+|u_s|)K_s^f+K_s^g]$.  Hence
\begin{equation*}
 \E^\Q\int_0^{t\wedge\alpha_n}
 \left|\ip{\mu_s}{\chi_Rr_s^\varepsilon
                         +\chi_Rs_s^\varepsilon}\right|\dd s
 \longrightarrow0.
 \label{eq:mollifier-drift-errors}
\end{equation*}
Expanding $C_s^{\varepsilon,R}$ as in
\eqref{eq:commutator-expanded}, using
\eqref{eq:mollifier-C0-uniform-time} for zeroth-order terms and
\eqref{eq:mollifier-local-time-integrated} for the first-order term,
gives
\begin{equation*}
 \E^\Q\int_0^{t\wedge\alpha_n}
 \left|\ip{\mu_s}{C_s^{\varepsilon,R}-C_s^R}\right|\dd s
 \longrightarrow0,
 \qquad
 C_s^R:=L_s(\chi_Rf_s)-\chi_RL_sf_s.
 \label{eq:mollifier-cutoff-commutator}
\end{equation*}
Here dominated convergence uses \eqref{eq:Kfg-integrability} and the
control-energy bound.  Finally,
\eqref{eq:mollifier-C0-uniform-time}--
\eqref{eq:mollifier-local-time-integrated} imply
\begin{equation*}
 \E^\Q\int_0^{t\wedge\alpha_n}
 \left|\ip{\mu_s}{\chi_R\bigl[(g_s^\varepsilon-g_s)
                   +(f_s^\varepsilon-f_s)h_s\bigr]}\right|^2\dd s
 \longrightarrow0.
 \label{eq:mollifier-stochastic-error}
\end{equation*}
It\^o's isometry and the preceding convergences yield
\begin{eqnarray*}
 \ip{\mu_{t\wedge\alpha_n}}{\chi_Rf_{t\wedge\alpha_n}}
 &={}&\ip{\mu_0}{\chi_Rf_0}
 +\int_0^{t\wedge\alpha_n}
   \ip{\mu_s}{\chi_R(g_s+f_sh_s)}^\top\dd Y_s
 \notag\\
& &+\int_0^{t\wedge\alpha_n}\ip{\mu_s}{C_s^R}\dd s.
 \label{eq:R-product}
\end{eqnarray*}

\medskip
\noindent\emph{Step 2: $R\uparrow\infty$ at fixed $n$.}
Since $a$ is symmetric,
\begin{equation}
 C_s^R=f_s b(\cdot,u_s)\cdot D\chi_R
 +\frac12f_sa(\cdot,u_s):D^2\chi_R
 +(a(\cdot,u_s)D\chi_R)\cdot Df_s.
 \label{eq:commutator-expanded}
\end{equation}
It is supported on $A_R=\{R\le|x|\le2R\}$ and
\begin{equation*}
 |C_s^R(x)|\le C\one_{A_R}(x)(1+|u_s|^2)K_s^f\br{x}^{\lambda+1}.
 \label{eq:commutator-growth}
\end{equation*}
The order $\lambda+1$ comes from
$a\sim\br{x}^2$, $D\chi_R\sim R^{-1}$, and
$Df\sim\br{x}^\lambda$.  For $s<\alpha_n$,
\eqref{eq:product-moment-before-stop} gives
\begin{align*}
 \left|\ip{\mu_s}{C_s^R}\right|
 &\le C(1+|u_s|^2)K_s^f
   \ip{\mu_s}{\one_{A_R}\br{\cdot}^{\lambda+1}}\\
 &\le Cc_n(1+|u_s|^2)K_s^f.
\end{align*}
The first line tends to zero for
$\dd\Q\otimes\dd s$-almost every $(\omega,s)$ because
$q_0\ge\lambda+1$, and the second is integrable by
\eqref{eq:Kfg-integrability} and the control-energy bound.  Thus
\begin{equation*}
 \E^\Q\left|\int_0^{t\wedge\alpha_n}
          \ip{\mu_s}{C_s^R}\dd s\right|\longrightarrow0.
 \label{eq:commutator-vanish}
\end{equation*}

For the boundary pairings, no $q_0$-moment at
$\alpha_n$ is used.  For each sample point the narrowly continuous
image $\{\mu_s:0\le s\le\alpha_n\}$ is a compact, hence uniformly
tight, family of finite measures.  Together with
\eqref{eq:f-bounded}, this gives, uniformly in $t$,
\[
 \ip{\mu_{t\wedge\alpha_n}}
       {(1-\chi_R)f_{t\wedge\alpha_n}}\longrightarrow0.
\]
The variables are bounded by $\norm{\gamma}_\infty c_n$ using
\eqref{eq:product-mass-at-stop}, so the convergence also holds in
$L^1(\Q)$; the initial pairing is analogous.
The two stochastic tails are controlled at the quadratic-variation level by
\begin{eqnarray*}
 \int_0^{\alpha_n}
 \left|\ip{\mu_s}{(1-\chi_R)g_s}\right|^2\dd s
 &\le& Cc_n^2 R^{2(\lambda-q_0)}
 \int_0^{\alpha_n}(K_s^g)^2\dd s,
 \label{eq:g-tail-qv}\\
 \int_0^{\alpha_n}
 \left|\ip{\mu_s}{(1-\chi_R)f_sh_s}\right|^2\dd s
 &\le& C\|\gamma\|_\infty^2c_n^2
 \left(TR^{2(1-q_0)}+kR^{-2q_0}\right).
 \label{eq:fh-tail-qv}
\end{eqnarray*}
Indeed, on $|x|\ge R$,
$\br{x}^\lambda\le R^{\lambda-q_0}\br{x}^{q_0}$, while integration against $\mu_s$ gives
\[
 \ip{\mu_s}{\one_{\{|x|\ge R\}}(1+|x|+|u_s|)}
 \le C\bigl(R^{1-q_0}+|u_s|R^{-q_0}\bigr)
 \ip{\mu_s}{\br{\cdot}^{q_0}}.
\]
Here \eqref{eq:product-moment-before-stop} is used
only for $s<\alpha_n$; the value at the single endpoint is irrelevant.
Because $q_0\ge\lambda+1$ and $q_0\ge2$, both right-hand sides tend
to zero in $L^1(\Q)$.  BDG therefore passes to the limit in the
stochastic integral.  We obtain
\begin{equation}
 \ip{\mu_{t\wedge\alpha_n}}{f_{t\wedge\alpha_n}}
 =\ip{\mu_0}{f_0}
 +\int_0^{t\wedge\alpha_n}\ip{\mu_s}{g_s+f_sh_s}^\top\dd Y_s.
 \label{eq:localized-final-product}
\end{equation}

\medskip
\noindent\emph{Step 3: remove the localizations and take expectations.}

Set
\[
 H_s^\mu:=\left\langle\mu_s,g_s+f_sh_s\right\rangle,
 \qquad
 \mathcal N_t^{(n)}:=\int_0^{t\wedge\alpha_n}(H_s^\mu)^\top\dd Y_s.
\]
The stochastic integrals are consistent under stopping and hence define a
continuous local martingale $\mathcal N$ on the stochastic interval
$[0,\theta)$, localized by $(\alpha_n)$.  For each sample point and every
$t<\theta$, one has $t<\alpha_n$ eventually.  The same compact-tightness
argument used in \Cref{lem:terminal-pairing-convergence} shows that
$t\mapsto\langle\mu_t,f_t\rangle$ is continuous on $[0,\theta)$.  Taking a
countable intersection for rational $t$ and then using continuity in the
stopped identities therefore shows that, simultaneously for all $t<\theta$,
\begin{equation}
 \left\langle\mu_t,f_t\right\rangle
 =\left\langle\mu_0,f_0\right\rangle+\mathcal N_t.
 \label{eq:product-before-terminal}
\end{equation}
The pathwise argument in \Cref{lem:terminal-pairing-convergence} applies to
an arbitrary sequence $t_j\uparrow\theta$: narrow continuity of $\mu$,
local-uniform time continuity of $f$, and $|f|\le\|\gamma\|_\infty$ give
\begin{equation*}
 \left\langle\mu_t,f_t\right\rangle
 \longrightarrow \left\langle\mu_\theta,\gamma\right\rangle
 \quad\text{as }t\uparrow\theta,
 \qquad\Q\text{-a.s.}
 \label{eq:product-terminal-pathwise-limit}
\end{equation*}
Thus \eqref{eq:product-before-terminal} gives a finite pathwise limit for
$\mathcal N_t$ as $t\uparrow\theta$.  The continuous-local-martingale
convergence theorem (equivalently, the Dambis--Dubins--Schwarz
representation) then implies
\[
 \langle\mathcal N\rangle_\theta
 =\int_0^\theta|H_s^\mu|^2\dd s<\infty
 \quad\Q\text{-a.s.}
\]
and $\mathcal N$ extends continuously to $\theta$ by
$\mathcal N_\theta:=\lim_{t\uparrow\theta}\mathcal N_t$.  Passing to the
limit in \eqref{eq:product-before-terminal} proves
\eqref{eq:measure-product} at the terminal time, and the stopped form gives
it for every $t\ge\theta$.  This also proves that the stochastic integral in
\eqref{eq:measure-product} is a localized continuous local martingale on the
closed stochastic interval $[0,\theta]$.

For every fixed $n$,
\eqref{eq:product-moment-before-stop},
\eqref{eq:Kfg-integrability}, \eqref{eq:f-bounded}, the linear growth
of $h$, and the control-energy bound give
\begin{align*}
 \E^\Q\int_0^{\alpha_n}
 \left|\ip{\mu_s}{g_s+f_sh_s}\right|^2\dd s
 \le Cc_n^2\left(
   \E^\Q\int_0^\theta(K_s^g)^2\dd s
   +\norm{\gamma}_\infty^2(T+k)\right)<\infty.
\end{align*}
Thus the stochastic integral stopped at $\alpha_n$ is a
square-integrable martingale.  Taking $t=T$ in
\eqref{eq:localized-final-product} and then expectations gives
\[
 \E^\Q\ip{\mu_{\alpha_n}}{f_{\alpha_n}}
 =\E^\Q\ip{\mu_0}{f_0}.
\]
Lemma~\ref{lem:terminal-pairing-convergence}, applied to
$\alpha_n\uparrow\theta$, gives
\[
 \ip{\mu_{\alpha_n}}{f_{\alpha_n}}
 \longrightarrow\ip{\mu_\theta}{\gamma}
 \quad\text{in }L^1(\Q).
\]
Letting $n\to\infty$ proves \eqref{eq:duality-expectation}.

\end{proof}

\subsection{Pathwise uniqueness of the Zakai equation}

\begin{theorem}[Pathwise uniqueness for the stopped Zakai equation]
\label{thm:zakai-uniqueness}
Suppose \Cref{ass:bspde-regularity} holds. For each $k$, the stopped Zakai equation has at most one admissible solution with a prescribed initial measure.  More precisely, on the same filtered probability space, with the same $Y$, $u$, and initial measure, any two admissible solutions are indistinguishable on $[0,\tau_k]$.  The same assertion holds on every $\mathbb F^Y$-stopping horizon $\vartheta\le\tau_k$.
\end{theorem}

\begin{proof}
Let $\mu^1,\mu^2$ be two admissible solutions with the same initial measure and set
$\bar\mu=\mu^1-\mu^2$.  Fix $t\in[0,T]$ and
$\varphi\in C_c^{n_*+1}(\R^d)$, and put
$\theta=t\wedge\tau_k$.  The random variable
\begin{equation*}
 \xi:=\sgn\ip{\bar\mu_\theta}{\varphi}
 \label{eq:sign-xi}
\end{equation*}
is bounded and $\FY_\theta$-measurable.  Hence
$\gamma(\omega,x)=\xi(\omega)\varphi(x)$ belongs to
$\mathscr T_\theta^{n_*+1}$.  Let $(f,g)$ solve the corresponding BSPDE \eqref{eq:main-bspde}.

Apply \eqref{eq:duality-expectation} to $\mu^1$ and $\mu^2$ and subtract.  Since their initial measures agree,
\[
 \E^\Q\ip{\bar\mu_\theta}{\xi\varphi}=0.
\]
By the definition of $\xi$,
\begin{equation*}
 \E^\Q\left|\ip{\bar\mu_{t\wedge\tau_k}}{\varphi}\right|=0.
 \label{eq:pairing-zero}
\end{equation*}
Thus the pairing vanishes almost surely.

Choose a countable measure-determining class
$\mathcal D_0\subset C_c^{n_*+1}(\R^d)$.  Intersecting the
full-probability events obtained above for $t\in\mathbb Q\cap[0,T]$ and
$\varphi\in\mathcal D_0$ gives
\[
 \left\langle\mu^1_{t\wedge\tau_k},\varphi\right\rangle
 =\left\langle\mu^2_{t\wedge\tau_k},\varphi\right\rangle
 \quad\text{for every }(t,\varphi)\in
 (\mathbb Q\cap[0,T])\times\mathcal D_0
\]
on one event of probability one.  The determining property and narrow
continuity of both stopped paths then imply
$\mu^1_{t\wedge\tau_k}=\mu^2_{t\wedge\tau_k}$ for every $t\in[0,T]$ on
that same event; hence the two solutions are indistinguishable on
$[0,\tau_k]$.

Finally, let $\vartheta\le\tau_k$ be an arbitrary $\mathbb F^Y$-stopping
time.  Stopping an admissible solution at $\vartheta$ preserves the stopped
weak equation, narrow continuity, and the class-$D$ mass condition, while
$\eta_N\wedge\vartheta$ supplies the required local $q_0$-moment sequence.
The same verification applies if the solution is originally specified only
on $[0,\vartheta]$.  Repeating the preceding argument with
\[
 \theta=t\wedge\vartheta,
 \qquad
 \xi=\sgn\left\langle
      \mu^1_{t\wedge\vartheta}-\mu^2_{t\wedge\vartheta},\varphi
      \right\rangle
\]
gives equality at every rational stopped time for the same class
$\mathcal D_0$.  Narrow continuity again upgrades this equality to
indistinguishability on $[0,\vartheta]$.

\end{proof}

\section{Pathwise uniqueness of the filtering equation}
\label{KS-equation}

\begin{definition}[Admissible stopped KS solution]
\label{def:admissible-KS}
An admissible solution of \eqref{eq:KS-common-Y} is an
$\mathbb F^Y$-adapted, narrowly continuous probability-measure-valued
process $\pi$ satisfying that equation for every $C_c^2$ test function and
for which, with
\[
 m_s:=\ip{\pi_s}{h(\cdot,u_s)},
\]
there are increasing $\mathbb F^Y$-stopping times
$\eta_n\uparrow\tau_k$ such that
\begin{equation}
 \sup_{0\le s\le\eta_n}
   \ip{\pi_s}{\br{\cdot}^{q_0}}
 +\int_0^{\eta_n}|m_s|^2\dd s<\infty
 \quad\text{a.s. for every }n.
 \label{eq:KS-local-m}
\end{equation}

\end{definition}

\begin{theorem}[Pathwise uniqueness for the stopped KS equation]
\label{thm:KS-uniqueness}
Suppose \Cref{ass:bspde-regularity} holds. For each $k$, equation \eqref{eq:KS-common-Y} has at most one admissible solution with a prescribed initial probability measure.
\end{theorem}

\begin{proof}

The whole unnormalization argument is carried out under the stopped reference
probability $\Q^k$.  On $[0,\tau_k]$ the process denoted by $Y$ is the
reference Brownian motion $Y^k$, and the Kushner--Stratonovich identity is the
same pathwise continuous-semimartingale identity under $\Q^k$ and under the
equivalent physical probability $\Pp$.  We transfer the final
indistinguishability statement back to $\Pp$ by equivalence.

Let $\pi^i$, $i=1,2$, be two solutions.  Put

\begin{equation*}
\begin{aligned}
 m_t^i&:=\ip{\pi_t^i}{h(\cdot,u_t)},\\
 R_t^i&:=\cE\left(
      \int_0^{\,\cdot}(m_s^i)^\top\dd Y_s\right)_t,
 \qquad
 \mu_t^i:=R_t^i\pi_t^i,\quad 0\le t<\tau_k.
\end{aligned}
 \label{eq:KS-unnormalization}
\end{equation*}

These processes are understood on their local
square-integrability intervals.  If
$\int_0^{\tau_k}|m_s^i|^2\dd s<\infty$, they are extended
continuously to $\tau_k$; no value of $R_{\tau_k}^i$ is otherwise
needed.  The stochastic exponentials are strictly positive continuous
local martingales on these local domains.  Apply It\^o's
formula to $R_t^i\ip{\pi_t^i}{\varphi}$ on any such interval.  With
$B_t^i(\varphi)=\ip{\pi_t^i}{\varphi h}-\ip{\pi_t^i}{\varphi}m_t^i$, the drift terms involving $m^i$ cancel:
\begin{align*}
 \dd\left(R_t^i\ip{\pi_t^i}{\varphi}\right)
 ={}&R_t^i\ip{\pi_t^i}{L_t\varphi}\dd t\\
 &+R_t^i\left(B_t^i(\varphi)
          +\ip{\pi_t^i}{\varphi}m_t^i\right)^\top\dd Y_t\\
 ={}&\ip{\mu_t^i}{L_t\varphi}\dd t
   +\ip{\mu_t^i}{\varphi h(\cdot,u_t)}^\top\dd Y_t.
\end{align*}
Thus $\mu^i$ solves the Zakai equation.

For $i=1,2$, choose increasing localizers
$\eta_n^i\uparrow\tau_k$ as in \eqref{eq:KS-local-m}.

Define on all of $[0,T]$
\[
 \bar A_t^i:=\int_0^{t\wedge\tau_k}|m_s^i|^2\dd s,
\]
where the value at $\tau_k$ is the monotone limit in
$[0,+\infty]$.  Every stochastic exponential below is first defined up to
the indicated energy stop and then held constant after that stop.

For $n\ge2$, define
\begin{align}
 \kappa_n
 &:=\inf\{t\in[0,T]:\bar A_t^1+\bar A_t^2\ge n\}\wedge\tau_k,
 \notag\\
 \delta_n
 &:=\inf\left\{t\ge0:
       R_{t\wedge\kappa_n}^1\notin(n^{-1},n)
       \quad\text{or}\quad
       R_{t\wedge\kappa_n}^2\notin(n^{-1},n)\right\}\wedge\kappa_n,
 \notag\\
 \rho_n&:=\eta_n^1\wedge\eta_n^2\wedge\delta_n.
 \label{eq:KS-common-localizers}
\end{align}
If $\kappa_n=\tau_k$, the increasing energy has a finite limit no
larger than $n$, so the local exponentials extend continuously through
$\kappa_n$; if $\kappa_n<\tau_k$, their brackets up to $\kappa_n$ are
bounded by $n$.  Thus $R^i$ is defined throughout the interval used
for $\delta_n$.  The quantities in
\eqref{eq:KS-common-localizers} are genuine stopping times, the
sequence is increasing, and $\rho_n\uparrow\tau_k$.  Indeed, on every
compact subinterval of $[0,\tau_k)$ the energies are finite and each
strictly positive continuous path $R^i$ is bounded above and away
from zero.

On $[0,\rho_n]$, continuity at first exits gives
$n^{-1}\le R^i\le n$ and $\bar A^1+\bar A^2\le n$.  Moreover,
\[
 \sup_{s\le\rho_n}
 \ip{\mu_s^i}{\br{\cdot}^{q_0}}
 \le n\sup_{s\le\eta_n^i}
       \ip{\pi_s^i}{\br{\cdot}^{q_0}}<\infty.
\]
Thus $\mu^i$ satisfies every condition of
\Cref{def:admissible-zakai} on the common horizon $\rho_n$, and the
common initial law gives $\mu_0^1=\mu_0^2$.
\Cref{thm:zakai-uniqueness} yields $\mu^1=\mu^2$ there.  Equality as
finite measures gives $R^1=R^2$ directly. Strict positivity then gives
$\pi^1=\pi^2$ on $[0,\rho_n]$.  Letting $n\to\infty$ and using narrow
continuity at $\tau_k$, first on a countable convergence-determining
class and then simultaneously in time, proves indistinguishability
on $[0,\tau_k]$.

\end{proof}

Finally, we remove the control-energy localization.

\begin{definition}[Consistent admissible stopped family]
\label{def:consistent-family}
A consistent admissible Zakai family consists of processes $(\mu^k)_{k\ge1}$ on the fixed physical stochastic basis such that $\mu^k$ is an admissible solution on $[0,\tau_k]$ under $\Q^k$, all members have the same initial measure, and, for every $\ell\ge k$,
\[
 \mu_t^\ell=\mu_t^k\quad\text{for all }t\le\tau_k,
 \qquad\Pp\text{-a.s.}
\]
The same equality holds under either local reference measure by equivalence and \eqref{eq:Q-consistency}.  A global admissible KS solution is interpreted under $\Pp$ and is required to satisfy Definition~\ref{def:admissible-KS} on every $[0,\tau_k]$.

\end{definition}

\begin{proposition}[Consistency of the actual stopped filters]
\label{prop:actual-consistent}
If the initial law satisfies \eqref{eq:initial-q0}, then the actual unnormalized filters $(\mu^k)_{k\ge1}$ defined by
\eqref{eq:mu-definition} form a consistent admissible stopped family.
\end{proposition}

\begin{proof}
Fix
$\ell\ge k$, a bounded Borel function $\varphi$, and a deterministic $t$.  Up to $t\wedge\tau_k$, the density processes agree,
$M_{t\wedge\tau_k}^\ell=M_{t\wedge\tau_k}^k$, and
$\Q^\ell=\Q^k$ on $\F_{\tau_k}$ by
\eqref{eq:Q-consistency}.  Hence, for every
$A\in\FY_{t\wedge\tau_k}$,
\begin{align*}
 &\E^{\Q^\ell}\!\left[
   \one_A M_{t\wedge\tau_k}^\ell
       \varphi(X_{t\wedge\tau_k})\right]=
 \E^{\Q^k}\!\left[
   \one_A M_{t\wedge\tau_k}^k
       \varphi(X_{t\wedge\tau_k})\right].
\end{align*}
The two local probabilities also have the same restriction to
$\FY_{t\wedge\tau_k}$.  Therefore their conditional expectations in
\eqref{eq:mu-definition} coincide:
\begin{equation*}
 \langle\mu_{t\wedge\tau_k}^\ell,\varphi\rangle
 =\langle\mu_{t\wedge\tau_k}^k,\varphi\rangle
 \quad\Q^k\text{-a.s.}
 \label{eq:actual-filter-consistency}
\end{equation*}
Choose a countable convergence-determining family of test functions and rational $t$, and then use the narrowly continuous versions constructed in
\Cref{prop:actual-admissible}.  This yields equality of the finite measures simultaneously for all $t\le\tau_k$.  Equivalence transfers the equality to $\Pp$ and to either local reference probability, proving consistency.
\end{proof}

\begin{theorem}[Uniqueness after patching the stopped equations]
\label{thm:global-uniqueness}
Suppose \Cref{ass:bspde-regularity} holds.  Then:
\begin{enumerate}[label=\textnormal{(\roman*)}]
\item for any prescribed initial measure, there is at most one consistent family whose $k$th member is an admissible stopped Zakai solution under $\Q^k$;
\item global admissible probability-measure-valued solutions of the Kushner--Stratonovich equation are pathwise unique on $[0,T]$;
\item if the prescribed initial probability measure has a finite moment of order $q_0=d+4$, the actual stopped unnormalized filters exist in the admissible class and form the consistent family of \Cref{prop:actual-consistent}.
\end{enumerate}
\end{theorem}

\begin{proof}
For every $k$, \Cref{thm:zakai-uniqueness,thm:KS-uniqueness} give uniqueness on $[0,\tau_k]$ under $\Q^k$.  Since $\Q^k$ and $\Pp$ are equivalent, indistinguishability is measure independent.  Two consistent Zakai families therefore agree member by member; this proves part~\textnormal{(i)} and is a uniqueness statement about a family of stopped equations, not about one global equation under one $\Q$.

For two global admissible Kushner--Stratonovich solutions, compare the same physical processes on every $[0,\tau_k]$.  The stopped uniqueness theorem gives equality on each such interval, and $\tau_k\uparrow T$ almost surely by the pathwise energy assumption.  Both solutions have the same prescribed initial law.  Equality at the terminal time follows on a countable convergence-determining class by narrow continuity as $t\uparrow T$ (in fact, pathwise finite control energy also implies that $\tau_k=T$ for every sufficiently large integer $k$ on each sample path).  This proves part~\textnormal{(ii)}.  Part~\textnormal{(iii)} is the combination of
\Cref{prop:actual-admissible,prop:actual-consistent}.
\end{proof}

\section{Stability of the nonlinear filters}
\label{sec:stability}

The preceding sections establish uniqueness for a fixed admissible filtering system.  We now record a complementary robustness result when the observation-adapted input, and hence both the signal and the observation model, vary.  Since the physical probabilities vary with the model, all objects are first realized on one common reference space.

Let $(\Omega,\F,\mathbb F,\Q)$ satisfy the usual conditions, and suppose that $W$ and $Y$ are independent Brownian motions of dimensions $r$ and $m$, respectively.  Let $u_n$ and $u$ be $\mathbb F^Y$-predictable processes, and let $X^n$ and $X$ solve
\begin{eqnarray*}
 \dd X_t^n
 &=&b(X_t^n,u_n(t))\dd t+\sigma(X_t^n,u_n(t))\dd W_t,
 \label{eq:stability-state-n}\\
 \dd X_t
 &=&b(X_t,u(t))\dd t+\sigma(X_t,u(t))\dd W_t.
 \label{eq:stability-state-limit}
\end{eqnarray*}
Put
\begin{equation*}
 H_t^n:=h(X_t^n,u_n(t)),
 \qquad
 H_t:=h(X_t,u(t)),
 \label{eq:stability-H}
\end{equation*}
and assume that their squared time integrals are finite almost surely.  Define the reference-to-physical likelihoods
\begin{eqnarray*}
 \mathsf{Z}_t^n
 &:=&\cE\left(\int_0^{\,\cdot}(H_s^n)^\top\dd Y_s\right)_t,
 \label{eq:stability-Zn}\\
 \mathsf{Z}_t
 &:=&\cE\left(\int_0^{\,\cdot}H_s^\top\dd Y_s\right)_t.
 \label{eq:stability-Z}
\end{eqnarray*}
For every $n$, suppose that $\mathsf{Z}^n$ is a uniformly integrable $\Q$-martingale, and define
\begin{equation*}
 \frac{\dd\Pp^n}{\dd\Q}\bigg|_{\F_T}=\mathsf{Z}_T^n.
 \label{eq:stability-Pn}
\end{equation*}
Under $\Pp^n$, the process $Y_t-\int_0^tH_s^n\dd s$ is Brownian, while $W$ remains Brownian.  Thus $\Pp^n$ realizes the physical filtering system corresponding to $u_n$.  Let
\begin{equation*}
 \pi_t^n:=\Pp^n(X_t^n\in\cdot\mid\FY_t).
 \label{eq:stability-filter-n}
\end{equation*}
Whenever $\mathsf{Z}$ is a true martingale, $\Pp$ and $\pi_t$ are defined analogously.

For probability measures on $\R^d$, write
\begin{equation*}
 d_{\mathrm{BL}}(\mu,\nu)
 :=\sup_{\|\varphi\|_\infty+\operatorname{Lip}(\varphi)\le1}
 \left|\ip{\mu-\nu}{\varphi}\right|.
 \label{eq:BL-distance}
\end{equation*}
This metric induces the topology of weak convergence.  We write $W_1$ for the first Wasserstein distance on $\mathcal P_1(\R^d)$.

\begin{theorem}[Stability of the normalized filters]
\label{thm:filter-stability}
Suppose that
\begin{align}
 X^n&\longrightarrow X
 &&\text{in $\Q$-probability in }C([0,T];\R^d),
 \label{eq:stability-X-convergence}\\
 \int_0^T|H_t^n-H_t|^2\dd t&\longrightarrow0
 &&\text{in $\Q$-probability},
 \label{eq:stability-H-convergence}
\end{align}
and that
\begin{equation}
 \{\mathsf{Z}_T^n:n\ge1\}
 \quad\text{is uniformly integrable under }\Q.
 \label{eq:stability-Z-UI}
\end{equation}
Then $\mathsf{Z}$ is a uniformly integrable $\Q$-martingale,
\begin{equation}
 \mathsf{Z}_T^n\longrightarrow\mathsf{Z}_T
 \quad\text{in }L^1(\Q),
 \qquad
 \|\Pp^n-\Pp\|_{\mathrm{TV}}\longrightarrow0,
 \label{eq:stability-density-TV}
\end{equation}
and, for every fixed $t\in[0,T]$,
\begin{equation}
 d_{\mathrm{BL}}(\pi_t^n,\pi_t)
 \longrightarrow0
 \quad\text{in $\Q$-probability}.
 \label{eq:stability-BL-conclusion}
\end{equation}
In particular, for every $f\in C_b(\R^d)$,
\begin{equation}
 \ip{\pi_t^n}{f}\longrightarrow\ip{\pi_t}{f}
 \quad\text{in $\Q$-probability and in }L^p(\Q)
 \quad\text{for every }1\le p<\infty.
 \label{eq:stability-Cb-conclusion}
\end{equation}
If, in addition, the corresponding physical state processes satisfy
\begin{equation}
 \sup_{n\ge1}\E^{\Pp^n}
 \left[\sup_{0\le s\le T}|X_s^n|^2\right]
 +\E^{\Pp}
 \left[\sup_{0\le s\le T}|X_s|^2\right]
 <\infty,
 \label{eq:stability-second-moment}
\end{equation}
then, for every fixed $t\in[0,T]$,
\begin{equation}
 W_1(\pi_t^n,\pi_t)\longrightarrow0
 \quad\text{in $\Q$-probability}.
 \label{eq:stability-W1-conclusion}
\end{equation}
Consequently, convergence in probability of the filter pairings remains valid for every $f\in C(\R^d)$ satisfying
\begin{equation}
 |f(x)|\le C_f(1+|x|).
 \label{eq:stability-linear-growth-test}
\end{equation}
The probability convergences in \eqref{eq:stability-BL-conclusion} and \eqref{eq:stability-W1-conclusion} also hold under the limiting physical probability $\Pp$, and in the sense that the probabilities of the corresponding exceptional events under $\Pp^n$ tend to zero.
\end{theorem}

\begin{proof}
We divide the proof into three steps.

\medskip
\noindent\emph{Step 1: convergence of the likelihoods.}
Set
\[
 A_n:=\int_0^T|H_s^n-H_s|^2\dd s,
 \qquad
 N_t^n:=\int_0^t(H_s^n-H_s)^\top\dd Y_s.
\]
For $\delta>0$, let
\[
 \rho_{n,\delta}
 :=\inf\left\{t\in[0,T]:
       \int_0^t|H_s^n-H_s|^2\dd s>\delta\right\}\wedge T.
\]
On $\{A_n\le\delta\}$, $N_T^n=N_{\rho_{n,\delta}}^n$.  Hence, by It\^o's isometry and Chebyshev's inequality, for every $\varepsilon>0$,
\begin{eqnarray*}
 \Q(|N_T^n|>\varepsilon)
 &\le& \Q(A_n>\delta)
   +\Q(|N_{\rho_{n,\delta}}^n|>\varepsilon)\notag\\
 &\le& \Q(A_n>\delta)+\frac{\delta}{\varepsilon^2}.
 \label{eq:stability-stochastic-integral}
\end{eqnarray*}
First let $n\to\infty$ and then $\delta\downarrow0$.  It follows that $N_T^n\to0$ in probability.  Moreover,
\begin{eqnarray*}
 \left|\int_0^T(|H_s^n|^2-|H_s|^2)\dd s\right|
 &\le& A_n^{1/2}
 \left(\left(\int_0^T|H_s^n|^2\dd s\right)^{1/2}
       +\left(\int_0^T|H_s|^2\dd s\right)^{1/2}\right)\notag\\
 &\le& A_n^{1/2}
 \left(A_n^{1/2}
       +2\left(\int_0^T|H_s|^2\dd s\right)^{1/2}\right),
 \label{eq:stability-energy-difference}
\end{eqnarray*}
which also converges to zero in probability.  Therefore
\[
 \log\mathsf{Z}_T^n-\log\mathsf{Z}_T
 =N_T^n-\frac12\int_0^T(|H_s^n|^2-|H_s|^2)\dd s
 \longrightarrow0
\]
in probability, and hence $\mathsf{Z}_T^n\to\mathsf{Z}_T$ in probability.  Uniform integrability and Vitali's theorem yield
\begin{equation}
 \E^\Q|\mathsf{Z}_T^n-\mathsf{Z}_T|\longrightarrow0.
 \label{eq:stability-Z-L1}
\end{equation}
Since $\E^\Q\mathsf{Z}_T^n=1$, we obtain $\E^\Q\mathsf{Z}_T=1$.  The positive local martingale $\mathsf{Z}$ is therefore a uniformly integrable martingale.  In particular,
\[
 \mathsf{Z}_t^n=\E^\Q[\mathsf{Z}_T^n\mid\F_t],
 \qquad
 \mathsf{Z}_t=\E^\Q[\mathsf{Z}_T\mid\F_t],
\]
and
\begin{equation}
 \E^\Q|\mathsf{Z}_t^n-\mathsf{Z}_t|
 \le\E^\Q|\mathsf{Z}_T^n-\mathsf{Z}_T|
 \longrightarrow0.
 \label{eq:stability-Zt-L1}
\end{equation}
The total-variation assertion in \eqref{eq:stability-density-TV} follows from
\[
 \|\Pp^n-\Pp\|_{\mathrm{TV}}
 =\frac12\E^\Q|\mathsf{Z}_T^n-\mathsf{Z}_T|.
\]

\medskip
\noindent\emph{Step 2: bounded continuous tests and weak convergence.}
For $f\in C_b(\R^d)$, put
\begin{equation*}
 \mathsf{N}_t^n(f)
 :=\E^\Q[\mathsf{Z}_t^n f(X_t^n)\mid\FY_t],
 \qquad
 \mathsf{D}_t^n:=\mathsf{N}_t^n(1),
 \label{eq:stability-ND-n}
\end{equation*}
and define $\mathsf{N}_t(f)$ and $\mathsf{D}_t$ analogously.  The Kallianpur--Striebel formula gives
\begin{equation}
 \ip{\pi_t^n}{f}=\frac{\mathsf{N}_t^n(f)}{\mathsf{D}_t^n},
 \qquad
 \ip{\pi_t}{f}=\frac{\mathsf{N}_t(f)}{\mathsf{D}_t}.
 \label{eq:stability-KS-ratio}
\end{equation}
By \eqref{eq:stability-X-convergence}, $f(X_t^n)\to f(X_t)$ in probability.  Together with \eqref{eq:stability-Zt-L1},
\begin{align*}
 &\E^\Q|\mathsf{Z}_t^n f(X_t^n)-\mathsf{Z}_t f(X_t)|\\
 &\quad\le
 \|f\|_\infty\E^\Q|\mathsf{Z}_t^n-\mathsf{Z}_t|
 +\E^\Q\big[\mathsf{Z}_t|f(X_t^n)-f(X_t)|\big]
 \longrightarrow0.
\end{align*}
The second term tends to zero because its bounded factor converges in probability under the finite measure $\mathsf{Z}_t\dd\Q$.  Conditional expectation is an $L^1$ contraction, so
\begin{equation*}
 \mathsf{N}_t^n(f)\longrightarrow\mathsf{N}_t(f),
 \qquad
 \mathsf{D}_t^n\longrightarrow\mathsf{D}_t
 \quad\text{in }L^1(\Q).
 \label{eq:stability-ND-L1}
\end{equation*}
Since $\mathsf{Z}_t>0$ almost surely, $\mathsf{D}_t=\E^\Q[\mathsf{Z}_t\mid\FY_t]>0$ almost surely.  Taking ratios proves \eqref{eq:stability-Cb-conclusion} in probability.  The difference is bounded by $2\|f\|_\infty$, so convergence also holds in every finite $L^p(\Q)$.

Choose a countable convergence-determining family $\{f_j:j\ge1\}\subset C_b(\R^d)$.  From every subsequence one may extract a further subsequence along which $\ip{\pi_t^n}{f_j}\to\ip{\pi_t}{f_j}$ a.s. for every $j$.  On the resulting common full-probability event, the deterministic measures $\pi_t^n$ converge weakly to $\pi_t$, and hence their bounded--Lipschitz distance tends to zero.  The subsequence criterion for convergence in probability proves \eqref{eq:stability-BL-conclusion}.

\medskip
\noindent\emph{Step 3: first Wasserstein convergence.}
Assume \eqref{eq:stability-second-moment}.  For $R>0$, set
\[
 \psi_R(x):=|x|\one_{\{|x|>R\}}.
\]

Although \eqref{eq:stability-KS-ratio} was stated for bounded continuous
tests, its underlying Bayes identity holds for every bounded Borel test.
Apply that identity first to $\psi_{R,M}:=\psi_R\wedge M$ and then use
conditional monotone convergence as $M\uparrow\infty$.  The physical
second-moment bound makes the limiting numerator integrable, since
\[
 \E^\Q[\mathsf Z_t^n\psi_R(X_t^n)]
 =\E^{\Pp^n}[\psi_R(X_t^n)]
 \le R^{-1}\E^{\Pp^n}|X_t^n|^2<\infty,
\]
uniformly in $n$, with the analogous estimate under $\Pp$.

For $a,\varepsilon>0$, Markov's inequality and the martingale property of $\mathsf{Z}^n$ give
\begin{eqnarray*}
 \Q\left(\ip{\pi_t^n}{\psi_R}>\varepsilon\right)
 &\le& \Q(\mathsf{D}_t^n<a)
 +\Q\left(
   \E^\Q[\mathsf{Z}_t^n\psi_R(X_t^n)\mid\FY_t]
   >a\varepsilon\right)\notag\\
 &\le& \Q(\mathsf{D}_t^n<a)
 +\frac{1}{a\varepsilon}
   \E^\Q[\mathsf{Z}_T^n\psi_R(X_t^n)]\notag\\
 &\le& \Q(\mathsf{D}_t^n<a)
 +\frac{C}{a\varepsilon R},
 \label{eq:stability-filter-tail}
\end{eqnarray*}
where $C$ is independent of $n$, because $\psi_R(x)\le |x|^2/R$.  Since $\mathsf{D}_t^n\to\mathsf{D}_t>0$ in probability,
\begin{equation*}
 \lim_{a\downarrow0}\limsup_{n\to\infty}
 \Q(\mathsf{D}_t^n<a)=0.
 \label{eq:stability-denominator-away-zero}
\end{equation*}
Consequently,
\begin{equation}
 \lim_{R\to\infty}\limsup_{n\to\infty}
 \Q\left(\ip{\pi_t^n}{\psi_R}>\varepsilon\right)=0.
 \label{eq:stability-uniform-tail}
\end{equation}
The same argument gives $\ip{\pi_t}{\psi_R}\to0$ in probability as $R\to\infty$.

Let $T_R:\R^d\to\overline B_R$ be the radial projection,
\[
 T_R(x)=x\quad\text{if }|x|\le R,
 \qquad
 T_R(x)=R\frac{x}{|x|}\quad\text{if }|x|>R.
\]
Then
\begin{align}
 W_1(\pi_t^n,\pi_t)
 \le{}&W_1((T_R)_\#\pi_t^n,(T_R)_\#\pi_t)\notag\\
 &+\int|x-T_R(x)|\,\pi_t^n(\dd x)
 +\int|x-T_R(x)|\,\pi_t(\dd x).
 \label{eq:stability-W1-truncation}
\end{align}
For fixed $R$, \eqref{eq:stability-BL-conclusion} implies weak convergence in probability of the two pushforward measures.  Since they are supported on the compact ball $\overline B_R$, weak convergence there is equivalent to convergence in $W_1$.  Hence the first term in \eqref{eq:stability-W1-truncation} tends to zero in probability for fixed $R$.  The remaining terms are bounded by the corresponding $\psi_R$-moments, which vanish in probability by \eqref{eq:stability-uniform-tail} and its limiting analogue.  First choosing $R$ large and then $n$ large proves \eqref{eq:stability-W1-conclusion}.

For completeness, from every subsequence of the random measures one may extract an almost surely $W_1$-convergent subsubsequence.  The deterministic characterization of $W_1$ convergence then gives convergence against every continuous function satisfying \eqref{eq:stability-linear-growth-test}; the subsequence criterion returns convergence in probability for the original sequence.  Since $\Pp$ is equivalent to $\Q$, convergence in $\Q$-probability implies convergence in $\Pp$-probability.  If $A_n$ denotes either exceptional event, then
\[
 \Pp^n(A_n)
 \le \Pp(A_n)+\|\Pp^n-\Pp\|_{\mathrm{TV}}\longrightarrow0,
\]
which gives the final assertion.
\end{proof}

\begin{remark}[Relation with convergence of the inputs]
\label{rem:stability-inputs}
The intrinsic assumption is \eqref{eq:stability-H-convergence}.  Under the hypotheses used elsewhere in this paper, $h$ is only Borel measurable in its input variable, so $\E\int_0^T|u_n(t)-u(t)|\dd t\to0$ does not by itself imply \eqref{eq:stability-H-convergence}.  A simple sufficient condition is
\[
 |h(x,v)-h(y,w)|\le C(|x-y|+|v-w|)
\]
and
\[
 \E^\Q\sup_{t\le T}|X_t^n-X_t|^2
 +\E^\Q\int_0^T|u_n(t)-u(t)|^2\dd t\longrightarrow0.
\]
If the controls are uniformly bounded, then their $L^1$ convergence in expectation implies the required $L^2$ convergence.  At a common control-energy stopping level, the entropy estimate of \Cref{prop:filtering-likelihood} also verifies \eqref{eq:stability-Z-UI}, provided that its entropy bound is uniform in $n$.
\end{remark}

\section{Concluding remarks}
\label{sec:conclusion}

The filtering equation produces the dual BSPDE rather than merely supplying an external application.  The entropy localization used to construct each stopped reference probability also supplies the parameterized likelihood required in the BSDE representation of the dual equation; the argument does not create an unstopped reference probability without extra integrability.  The bounded first component and the polynomially weighted estimates for both $f$ and $g$ then make the measure-valued product formula rigorous and close the uniqueness argument.

The stability theorem further shows that, once a common global reference probability is available and the likelihoods are uniformly integrable, the normalized filter depends continuously on the state and on the induced observation drift: weak stability follows from $L^2$ convergence of the observation drifts, while a uniform physical second moment upgrades the conclusion to $W_1$ stability.

The present model keeps the signal and observation noises independent and the observation covariance equal to the identity.  Correlated noises generate first-order spatial terms in the Zakai noise operator and corresponding $Dg$ couplings in the dual BSPDE; that problem is analytically different rather than a notational vector extension.  The current result is designed as the filtering infrastructure for subsequent partially observed control, game, and stopping problems with observation-adapted random inputs and unbounded observation drift.

\section*{Declaration on the Use of AI Tools}

This work grew out of earlier research by some of the authors on partially observed stochastic control \cite{Sunw}. Our aim to study nonlinear stochastic control problems with unbounded observation coefficients led us to the filtering problem addressed here. The authors formulated the research problem, developed the mathematical approach and key arguments, and established the principal results. ChatGPT Pro 5.6 helped accelerate our work. Subsequently, GPT-6 Astra was used to conduct a separate review of the manuscript and suggest corrections. The authors independently verified every proof suggested by these tools, checked all incorporated revisions, and take full responsibility for the paper.

\appendix
\section{A backward SPDE with unbounded coefficients}
\label{sec:duality-problem}

The following BSPDE
\begin{equation}
 \begin{cases}
 \dd f_s(x)
 =-\left[L_sf_s(x)+h(x,u_s)^\top g_s(x)\right]\dd s
   +g_s(x)^\top\dd Y_s,&0\le s\le\theta,\\
 f_\theta(x)=\gamma(x),
 \end{cases}
 \label{eq:main-bspde}
\end{equation}
is studied in \cite{XXY}. We present its  definition and some main results here for the convenience of the reader.

We work on the reference space $(\Omega,\F,\mathbb F,\Q)$ introduced in the main body of the paper.  Thus $Y$ is a Brownian motion of dimension $m$, and $\theta\le \tau_k$ is a stopping time.

 Let $\mathscr G_\theta^{n+1}$ be the class of jointly measurable fields $\gamma:\Omega\times\R^d\to\R$ such that
\begin{enumerate}[label=\textnormal{(\roman*)}]
\item $\gamma(\cdot,x)$ is $\FY_\theta$-measurable for every $x$;
\item $x\mapsto\gamma(\omega,x)$ belongs to $C_b^{n+1}(\R^d)$ for almost every $\omega$;
\item the spatial derivatives admit jointly measurable versions and
\begin{equation*}
 \norm{\gamma}_{\mathscr G_\theta^{n+1}}
 :=\esssup_{\omega}\max_{|\alpha|\le n+1}
       \sup_{x\in\R^d}|D_x^\alpha\gamma(\omega,x)|<\infty.
 \label{eq:general-terminal-class}
\end{equation*}
\end{enumerate}
The smaller finite-rank class
\begin{equation*}
 \mathscr T_\theta^{n+1}
 :=\left\{
   \gamma(\omega,x)=\sum_{j=1}^N\xi_j(\omega)\phi_j(x):
   \begin{array}{l}
     N<\infty,\ \xi_j\in L^\infty(\Omega,\FY_\theta,\Q),\\
     \phi_j\in C_c^{n+1}(\R^d)
   \end{array}
 \right\}
 \subset\mathscr G_\theta^{n+1}
 \label{eq:terminal-class}
\end{equation*}
is the only terminal class needed for the  uniqueness proof in this article.  In particular, $\xi\phi$ with arbitrary bounded $\FY_\theta$-measurable $\xi$ is allowed.

Set
\begin{equation*}
 n_*:=\left\lfloor\frac d2\right\rfloor+3.
 \label{eq:n-star}
\end{equation*}
This is the smallest integer satisfying
\begin{equation*}
 n_*>\frac d2+2.
 \label{eq:n-star-embedding}
\end{equation*}
The strict inequality is exactly what is needed to obtain a $C_x^2$ version from an $H_x^{n_*}$ estimate.

\begin{assumption}[Spatial regularity used only for the dual BSPDE]
\label{ass:bspde-regularity}
In addition to \Cref{ass:filtering}, the positive-order spatial derivatives
\[
 D_x^\alpha b,\qquad D_x^\alpha\sigma,\qquad D_x^\alpha h
\]
exist and are bounded uniformly on $\R^d\times U$ for every multi-index
$\alpha$ with $1\le|\alpha|\le n_*+1$.
\end{assumption}

The extra derivative at order $n_*+1$ is used in the classical stochastic-flow construction of $n_*$ spatial derivatives.  

For $\ell>0$ and an integer $n\ge0$, define
\begin{equation*}
 \norm{v}_{H_\ell^n(\R^d;\R^q)}^2
 :=\sum_{a=1}^q\sum_{|\alpha|\le n}
   \int_{\R^d}|D^\alpha v^a(x)|^2\br{x}^{-2\ell}\dd x,
 \label{eq:weighted-Hn}
\end{equation*}
with the scalar convention when $q=1$.  We suppress the domain and target when they are clear.

\begin{definition}[Bounded weighted Sobolev solution]
\label{def:bspde-solution}
Let $\theta\le\tau_k$ be an $\mathbb F^Y$-stopping time and
$\gamma\in\mathscr G_\theta^{n_*+1}$.  A pair $(f,g)$ is a bounded weighted Sobolev solution of \eqref{eq:main-bspde}
if, for some $\ell>0$,
\[f\text{ is }\mathbb F^Y\text{-adapted},\quad g\text{ is }\mathbb F^Y\text{-predictable},\quad
   f\in L^\infty(\Omega\times[0,\theta]\times\R^d),
 \label{eq:solution-meas}\]
 \begin{equation}\E^\Q\esssup_{s\le\theta}\norm{f_s}_{H_\ell^{n_*}}^2
   +\E^\Q\int_0^\theta\norm{g_s}_{H_\ell^{n_*}(\R^d;\R^m)}^2\dd s
   <\infty,
 \label{eq:solution-norm}
\end{equation}
and, for every $\varphi\in C_c^\infty(\R^d)$ and $t\in[0,T]$,
\begin{align}
 \ip{f_{t\wedge\theta}}{\varphi}
 ={}&\ip{\gamma}{\varphi}
 +\int_{t\wedge\theta}^{\theta}\ip{f_s}{L_s^*\varphi}\dd s
 \notag\\
 &+\int_{t\wedge\theta}^{\theta}
      \ip{h(\cdot,u_s)^\top g_s}{\varphi}\dd s
 -\int_{t\wedge\theta}^{\theta}\ip{g_s}{\varphi}^\top\dd Y_s.
 \label{eq:weak-bspde}
\end{align}
Here $L_s^*$ is the distributional adjoint of $L_s$.
\end{definition}

\begin{theorem}
\label{thm:main-bspde}
Suppose \Cref{ass:bspde-regularity} hold.  Fix $k\in\N$, an $\mathbb F^Y$-stopping time $\theta\le\tau_k$, and
$\gamma\in\mathscr G_\theta^{n_*+1}$.  Set
\begin{equation}
 \ell_0:=\frac d2+2.
 \label{eq:ell-zero}
\end{equation}

Then \eqref{eq:main-bspde} admits a unique bounded weighted Sobolev
solution $(f,g)$ with $\ell=\ell_0$ in \eqref{eq:solution-norm}.  The
component $f$ is $\mathbb F^Y$-adapted/predictable and has a single jointly
measurable, locally space--time continuous representative for which,
on one event of full $\Q$-probability,
\begin{equation}
 |f_s(x)|\le\norm{\gamma}_{L^\infty(\Omega\times\R^d)},
 \qquad 0\le s\le\theta,\quad x\in\R^d.
 \label{eq:f-bounded}
\end{equation}

For every
\begin{equation*}
 \lambda>d+2
 \label{eq:lambda-threshold}
\end{equation*}
and every multi-index $\alpha$ with $|\alpha|\le2$,
\begin{align}
 \E^\Q\Bigg[
 &\esssup_{0\le s\le\theta}\sup_{x\in\R^d}
     \frac{|D^\alpha f_s(x)|^2}{\br{x}^{2\lambda}}
 +\int_0^\theta\sup_{x\in\R^d}
     \frac{|D^\alpha g_s(x)|^2}{\br{x}^{2\lambda}}\dd s
 \Bigg]<\infty.
 \label{eq:weighted-sup-2}
\end{align}
In addition, $f$ admits a version that is continuous in time with values in $C(B_R)$ for every $R<\infty$.  Consequently, in \eqref{eq:weighted-sup-2} with $\alpha=0$ the essential time supremum may be replaced by the ordinary supremum.  The numerical exponents $\ell_0$ and the threshold $d+2$ are independent of $k$; only the constants in the estimates depend on the stopping level.
\end{theorem}

The following proposition supplement the main theorem above for the solution of the BSPDE, which will be useful in the proof of
\Cref{thm:measure-product}.

\begin{proposition}[Local time continuity of the first component]
\label{prop:local-time-continuity}
Every weighted Sobolev solution has a single jointly measurable,
adapted modification, fixed simultaneously on all integer balls,
such that, for every $R<\infty$,
\begin{equation}
 f\in C\bigl([0,\theta];H^{n_*-2}(B_R)\bigr)
 \hookrightarrow C\bigl([0,\theta];C(\overline B_R)\bigr),
 \label{eq:local-time-continuity}
\end{equation}
where the process is extended constantly after $\theta$ when
convenient.  This modification remains jointly measurable and
$\mathbb G^k$-adapted.  For the constructed solution, fix this
representative once and for all and use it henceforth.
\end{proposition}

\end{document}